\documentclass[11pt]{amsart}
\usepackage{amsmath, amssymb, amsthm, bm, mathtools, enumitem, caption, centernot}
\usepackage[pagebackref]{hyperref}
\usepackage{hyperref}

\usepackage[noabbrev,capitalize]{cleveref}
\usepackage{adjustbox}
\crefname{equation}{}{}
\usepackage{ytableau}
\usepackage{graphicx, tikz}
\usepackage[textheight=8.65in, textwidth=6.8in]{geometry}
\usepackage{lscape}

\usepackage{microtype}

\newtheorem{theorem}{Theorem}[section]
\newtheorem{lemma}[theorem]{Lemma}
\newtheorem{corollary}[theorem]{Corollary}
\newtheorem{proposition}[theorem]{Proposition}

\newtheorem*{conjecture*}{Conjecture}

\newtheorem{question}[theorem]{Question}

\theoremstyle{definition}
\newtheorem{definition}[theorem]{Definition}

\theoremstyle{remark}
\newtheorem*{remark}{Remark}
\newtheorem*{example}{Example}

\numberwithin{equation}{section}

\newcommand{\N}{\mathbb N}

\def\lp{\left(}
\def\rp{\right)}

\newcommand{\HH}{\mathbb H}

\newcommand{\cE}{\mathcal E}

\newcommand{\cM}{\mathcal M}

\newcommand{\cS}{\mathcal S}

\newcommand{\C}{\mathbb C}

\newcommand{\Q}{\mathbb Q}

\newcommand{\Z}{\mathbb Z}

\newcommand{\cQM}{\widetilde{\mathcal{M}}}

\newcommand{\lcm}{\mathrm{lcm}}

\title[$q$-multiple zeta values, quasimodular forms and partitions]{Higher level $q$-multiple zeta values with applications to quasimodular forms and partitions}
\date{\today}
\author{William Craig}
\address{Department of Mathematics, United States Naval Academy, 572C Holloway Road
Mail Stop 9E. Annapolis, MD 21402}
\email{wcraig@usna.edu}

\begin{document}

\begin{abstract}
In recent years, the generalized sum-of-divisor functions of MacMahon have been unified into the algebraic framework of $q$-multiple zeta values. In particular, these results link partition theory, quasimodular forms, $q$-multiple zeta values, and quasi-shuffle algebras. In this paper, we complete this idea of unification for higher levels, demonstrating that any quasimodular form of weight $k \geq 2$ and level $N$ may be expressed in terms of the $q$-multiple zeta values of level $N$ studied algebraically by Yuan and Zhao. We also give results restricted to $q$-multiple zeta values with integer coefficients, and we construct completely additive generating sets for spaces of quasimodular forms and for quasimodular forms with integer coefficients. We also provide a variety of computational examples from number-theoretic perspectives that suggest many new applications of the algebraic structure of $q$-multiple zeta values to quasimodular forms and partitions.
\end{abstract}

\maketitle

\section{Introduction and Statement of Results}

Multiple zeta values are natural generalizations of the famous integral values of Riemann's zeta function and are defined by
\begin{align*}
    \zeta\lp k_1, \dots, k_a \rp := \sum_{0 < n_1 < n_2 < \dots < n_a} \dfrac{1}{n_1^{k_1} n_2^{k_2} \cdots n_a^{k_a}}.
\end{align*}
These real numbers have appeared frequently across mathematics and physics. Their modern study was initiated by \cite{Hoffman97,Zagier94}, and the reader can consult \cite{BurgosGilFresan,Zhao20} for surveys on topics related to multiple zeta values and the helpful webpage of Hoffman \cite{HoffmanWeb} for a very detailed list of papers related to multiple zeta values. The algebraic structure of the vector space generated by multiple zeta values (indeed, even the fact that this space is an algebra at all) is very deep and beautiful. This theory continues to be a highly active area of research with many interesting new developments and major unsolved conjectures.

In this work, we consider $q$-generalizations of multiple zeta values, which we simply call $q$-multiple zeta values (see Section \ref{S: q-MZV} for constructions), from a number-theoretic point of view. The starting point for us is MacMahon's work on generalized divisor sums \cite{MacMahon}. These sums are defined by the $q$-series
\begin{align*}
    \mathcal U_a(q) := \sum_{0 < m_1 < m_2 < \dots < m_a} \dfrac{q^{m_1 + m_2 + \dots + m_a}}{\lp 1 - q^{m_1} \rp^2 \lp 1 - q^{m_2} \rp^2 \cdots \lp 1 - q^{m_a} \rp^2}.
\end{align*}
It was conjectured by MacMahon and proven by Andrews and Rose \cite{AndrewsRose} that the $\mathcal U_a$ are (mixed weight) quasimodular forms (see Section \ref{S: Nuts and Bolts} for formal definitions). In other language, they proved that the $\mathcal U_a(q)$ can be represented in terms of the weight $2k \geq 2$ Eisenstein series
\begin{align*}
    E_{2k}(z) := 1 - \dfrac{4k}{B_{2k}} \sum_{n \geq 1} \sigma_{2k-1}(n) q^n, \ \ \ q := e^{2\pi i z},
\end{align*}
where $B_{2k}$ are the Bernoulli numbers and $\sigma_\ell(n) := \sum_{d|n} d^\ell$ is the $\ell$th divisor sum function. One can see immediately connections between zeta values and these Eisenstein series through Euler's formulas for even zeta values, which represent the terms $\frac{4k}{B_{2k}}$ in terms of the special value $\zeta(2k)$ of the Riemann zeta function, and thus one can represent these zeta values in terms of asymptotics of $\sum_{n \geq 1} \sigma_{2k-1}(n) q^n$ as $q \to 1$ from inside the unit disk. Generalizing this theme, MacMahon's generalized sum-of-divisor functions can likewise be seen to represent the multiple zeta values as $q \to 1$ from inside the unit disk. From this point of view, then, we view these $q$-sums as a natural part of the algebraic study of multiple zeta values.

We may, however, equally well view these $q$-series from the modular perspective after the work of Andrews and Rose. Modular forms and quasimodular forms are foundational objects in number theory; see the texts \cite{123,CohenStromberg,DiamondShurman,Ono,Shimura} for many examples. The quasimodularity of MacMahon's series has lead to much study from the number-theoretic perspective in recent years \cite{AAT,AOS,AndrewsRose,Bachmann24,BachmannKuhn,BCIP,ChoieLee,CraigIttersumOno,JPS,Larson,Milas,OnoSingh,Rhoades,Rose}. In particular, it is easy to show that the Fourier coefficients of $\mathcal{U}_a(q)$ can be expressed as very natural sums of products of the multiplicities of certain classes of {\it partitions}. A partition of $n$ is a list of part sizes and multiplicities, i.e. $\lp 1^{m_1}, \dots, k^{m_k} \rp$, such that $\sum j m_j = n$. Partitions connect a wide array of topics from number theory, physics, representation theory and combinatorics, and in particular produce very nice modular or $q$-hypergeometric structures. See Andrews' excellent book \cite{AndrewsBook} for many more facts about partitions. Thus, the study of such $q$-series from a number-theoretic lens is very natural as well.

The previously mentioned algebraic point of view (i.e. the study of $q$-multiple zeta values) and the number-theoretic point of view (i.e. divisor sums, partition sums and modular forms) are united in many papers, leading up to the very important work of Bachmann and K\"{u}hn \cite{BachmannKuhn}. It is shown by Bachmann and K\"{u}hn that a generalization of MacMahon's series form spaces of $q$-multiple zeta values with special differential and algebraic structure, and they demonstrate that the entire algebra of quasimodular forms sits within this algebra of $q$-multiple zeta values. For specifics of their work, we refer the reader to Section \ref{S: q-MZV Level One}. These number-theoretic connections are deep and have, for instance, produced connections between $q$-multiple zeta values with weakly holomorphic modular forms \cite{JPS,OnoSingh}, partition congruences \cite{AOS}, higher Appell functions \cite{BCIP}, and prime numbers \cite{CraigIttersumOno}. In many of these connections, it is natural from the number-theoretic point of view to interpret the multiple divisor sums as sums over partitions into a fixed number of part sizes; thus this theory achieves a unification of $q$-multiple zeta values, partitions, and quasimodular forms. Of particular interest to us will be a generalization of a result of \cite{CraigIttersumOno} regarding producing new types of additive spanning sets for spaces of quasimodular forms.

The goal of our paper is to lift this unification to level $N$. The relevant algebraic point of view has been formulated by Yuan and Zhao \cite{YuanZhao}, who construct a natural level $N$ generalization of Bachmann and K\"{u}hn's work (see Section \ref{S: q-MZV Level N} for details). More specifically, their work generalizes in a natural way certain multiple zeta values of level $N$ constructed by restricting summands to specified congruence classes modulo $N$. Such multiple zeta values of level $N$ themselves are of significant interest in the multiple zeta value literature and have many known applications \cite{GKZ,KanekoTasaka,LalinRoy,XuZhao,Seki}. To achieve their natural $q$-generalizations, Yuan and Zhao twist the Bachmann-K\"{u}hn $q$-multiple divisor sums by $N$th root of unity\footnote{Really, any primitive $N$th root of unity will do; we make a fixed choice for convenience. In fact, the change of root of unity $\zeta_N \mapsto \zeta_N^m$ can be accounted for by a change of colors $c_j \mapsto mc_j$. Thus, Yuan and Zhao's space of $q$-multiple divisor sums of level $N$ is fully generated by any particular choice of primitive $N$th root of unity.} $\zeta_N := e^{2\pi i/N}$, yielding {\it $q$-multiple zeta values at level $N$}. The canonical examples of such functions, defined for $N \geq 1$, $\vec{k} = \lp k_1, \dots, k_a \rp \in \Z_{\geq 1}^a$, and $\vec{c} = \lp c_1, \dots, c_a \rp \in \lp \Z/N\Z \rp^a$, will come from the {\it twisted multiple divisor sums} defined by
\begin{align} \label{Eq: Twisted multiple divisor sum}
    \sigma_{\vec{k},\vec{c}}(n) := \sum_{\substack{m_1 n_1 + \dots + m_a n_a = n \\ 0 < n_1 < n_2 < \dots < n_a}} \zeta_N^{c_1 m_1 + c_2 m_2 + \dots + c_a m_a} m_1^{k_1} m_2^{k_2} \cdots m_a^{k_a}.
\end{align}
We call $\vec{k}$ the {\it weights}, $\vec{c}$ the {\it colors}, and $a$ the {\it depth} of the special $q$-multiple zeta values
\begin{align*}
    \mathcal U_{\vec{k}; \vec{c}; N}(q) := \sum_{n \geq 1} \sigma_{\vec{k},\vec{c}}(n) q^n.
\end{align*}
As in the literature on MacMahon's functions, we can interpret these twisted sums in terms of partitions of $n$ into exactly $a$ part sizes; thus we have from Yuan and Zhao a unification of partition-theoretic and $q$-multiple zeta value concepts.

Motivated by the quasimodularity theorems of Andrews and Rose \cite{AndrewsRose} and Bachmann and K\"{u}hn \cite{BachmannKuhn}, as well as follow-up work of Rose \cite{Rose} and Larson \cite{Larson} involving higher level quasimodular forms, we consider the space of quasimodular forms of level $N$ in the context of twisted multiple divisor sums and $q$-multiple zeta values of level $N$. Using the machinery of Yuan and Zhao as well as algebraic facts about quasimodular forms, we prove the following theorem (see Section \ref{S: Nuts and Bolts} for notation, which is standard in the literature on modular forms).

\begin{theorem} \label{T: Main Theorem, Cyclotomic}
    Let $f$ be a (mixed weight) quasimodular form of (maximal) weight $k \geq 1$ for the congruence subgroup $\Gamma(N)$, $N \geq 3$. Then if $k \geq 2$ or if $k=1$ and $f$ belongs to the Eisenstein space $\mathcal E_1\lp \Gamma(N) \rp$, $f$ can be represented as a linear combination of twists by $N$th roots of unity of $q^{1/N}$-multiple zeta values of level $N$ of highest weight $k$, highest depth $k$, and level $N$.
\end{theorem}

\begin{remark}  
    We make the following remarks on Theorem \ref{T: Main Theorem, Cyclotomic}:
    \begin{enumerate}
        \item The analogous result for any congruence subgroup of $\mathrm{SL}_2\lp \Z \rp$ follows immediately by the surjectivity of the trace operator between $\Gamma(N)$ and such subgroups. In particular, the result also follow $N=2$ by taking traces $\Gamma(4) \to \Gamma(2)$. All other results of this paper follow for level 2 similarly.
        \item As was pointed out to the author be J-W. van Ittersum, this result is optimal across weights since $q$-multiple zeta values of homogeneous weight one are linear combinations of Eisenstein series. It would be interesting to see whether any generalization of classical $q$-multiple zeta values might suffice to cover these outliers, as cusp forms of weight one are interesting objects in their own right \cite{DeligneSerre}.
        \item The main proof we give for Theorem \ref{T: Main Theorem, Cyclotomic} has a computational downfall that in practice it produces lengthy very linear combination of $q$-series. We therefore also describe an alternative computational approach in Section \ref{S: Nuts and Bolts} which allows for fewer products if one allows the level to grow, once formulas for higher weight Eisenstein series are found.
    \end{enumerate}
\end{remark}

Because every $q$-multiple zeta value can be interpreted as a sum over partitions into a fixed number of part sizes, Theorem \ref{T: Main Theorem, Cyclotomic} also implies that properties of coefficients of (quasi)modular forms of any level can be expressed in terms of properties of these classes of partitions. This concept fits snugly within the philosophy of multiplicative partition theory laid out by Schneider \cite{Schneider17,SchneiderThesis,Schneider16,SchneiderSills}. Schneider's theory also interlaces quasimodular forms, partitions, and multiple zeta functions, but in a completely different manner. The mutual interaction of these approaches should be studied further, as for example has been done for instance in \cite{BachmannIttersum} with level one. Using the quasi-shuffle structure of $q$-multiple zeta values at level $N$, we can obtain even more explicit information for modular forms of weight at least two.

As will be see in Section \ref{S: Nuts and Bolts}, surjectivity results for quasimodular forms at level $N$ cannot avoid operating in the cyclotomic field $\Q\lp\zeta_N\rp$ since many forms of level $N$ have coefficients in that field. We can, however, investigate the subalgebra of quasimodular forms with integral coefficients. Fortunately, the method by which Theorem \ref{T: Main Theorem, Cyclotomic} is proven is very amenable to this alternative situation, and with very little extra work we can achieve a result for $\Z[[q]]$-forms.

\begin{theorem} \label{T: Main Theorem, Integral}
    Let $f$ be a (mixed weight) quasimodular form of highest weight $k \geq 1$ for $\Gamma(N)$, $N \geq 3$, with integral Fourier coefficients. Then if $k \geq 2$ or if $k=1$ and $f$ resides in the Eisenstein space, then $f$ can be represented as a linear combination of $q$-multiple zeta values of level $N$ with integral coefficients having highest weight $k$, highest depth $k$, and level $N$.
\end{theorem}

\begin{remark}
    The proof shows more; indeed, all such forms arise from traces of the Galois group $\mathrm{Gal}\lp \Q\lp \zeta_N \rp/\Q \rp$ acting on certain $q$-multiple zeta values of level $N$.
\end{remark}

The results above already summarize a new and important number-theoretic structure arising from $q$-multiple zeta values; namely, all quasimodular forms of weights $k \geq 2$ are in fact $q$-multiple zeta values. This fact opens up many questions about applications; for instance, how might one utilize the algebraic properties of spaces of $q$-multiple zeta values to study quasimodular forms? One application on which we wish to focus is the production of new kinds of bases for spaces of quasimodular forms. The canonical basis for quasimodular forms at level one is formed by products of Eisenstein series of the form $E_2^a E_4^b E_6^c$. From a number-theorists point of view, this basis can be difficult to work with because of the number-theoretic difficulties of convolution sums arising from the multiplication of $q$-series. In \cite{CraigIttersumOno}, it was realized that one can produce a spanning set for the entire vector space of quasimodular forms in level one using $q$-multiple zeta values (without any convolution products) by using certain symmetric sums and properties of quasi-shuffle algebras. In this work, we generalize this phenomenon to level $N$, demonstrating linear generation for spaces of quasimodular forms by generalizing the ideas of symmetric sums.

We will begin by defining the required terms. For any $a \geq 1$, define the group $S_a^\pm := S_a \times \lp \Z/2\Z \rp^a$, where $S_a$ is the symmetric group of order $a$. Here, we view $\Z/2\Z = \{ 1,-1 \}$ as a multiplicative group. For $\sigma \in S_a$ and $s = (s_1, \dots, s_a) \in \lp \Z/2\Z \rp^a$, we define several group actions. In particular, for $\vec{k} = (k_1, \dots, k_a) \in \Z_{\geq 0}^a$, $\vec{c} = (c_1, \dots, c_a) \in \lp \Z/N\Z \rp^a$, we define the standard actions
\begin{align*}
    \sigma\lp \vec{k}, \vec{c} \rp = \lp \sigma \vec{k}, \sigma \vec{c} \rp = \lp \lp k_{\sigma(1)}, \dots, k_{\sigma(a)} \rp, \lp c_{\sigma(1)}, \dots, c_{\sigma(n)} \rp \rp
\end{align*}
and
\begin{align*}
    s \lp \vec{k}, \vec{c} \rp = \lp -1 \rp^{\varepsilon_s(\vec{k},\vec{c})} \lp \vec{k}, s \vec{c} \rp = \lp -1 \rp^{\varepsilon_s(\vec{k},\vec{c})} \lp \vec{k}, (s_1 c_1, \dots, s_a c_a) \rp, \ \ \ \varepsilon_s(\vec{k},\vec{c}) := \sum_{j=1}^a \frac{1-s_j}{2} (k_j+1)
\end{align*}
along with the composite action of $S_a^\pm$ given as
\begin{align*}
    \lp \sigma, s \rp \cdot \lp \vec{k}, \vec{c} \rp = s \lp \sigma \lp \vec{k}, \vec{c} \rp \rp.
\end{align*}
We then define the traceforms by acting on $\mathcal{U}_{\vec{k},\vec{c}}$ by acting on subscripts linearly, i.e.
\begin{align*}
    \mathcal{U}_{\vec{k},\vec{c}}^{\mathrm{sym}}(q) := \sum_{(\sigma,s) \in S_a \times \lp \Z/2\Z \rp^a} \lp \sigma, s \rp \cdot \mathcal{U}_{\vec{k},\vec{c}}(q) = \sum_{(\sigma,s) \in S_a \times \lp \Z/2\Z \rp^a} (-1)^{\varepsilon_s(\vec{k},\vec{c})} \mathcal{U}_{\sigma \vec{k}, s \sigma \vec{c}}(q).
\end{align*}
We extend this action to spaces generated by the $q$-multiple zeta values $\mathcal{U}_{\vec{k};\vec{c};N}(q)$ by linearity.

\begin{theorem} \label{T: Main Theorem, Symmetry}
    Let $N \geq 1$, $\vec{k} \in \Z_{\geq 0}^a$ and $\vec{c} \in \lp \Z/N\Z \rp^a$ be given. Then the following are true:
    \begin{enumerate}
        \item The forms $\mathcal{U}^{\mathrm{sym}}_{\vec{k},\vec{c}}(q^{1/N})$ are quasimodular forms of level $N$, and the space of such forms is closed as an algebra.
        \item The collection of all twists of the forms $\mathcal{U}^{\mathrm{sym}}_{\vec{k},\vec{c}}(q^{1/N})$ generates additively the space of all quasimodular forms of level $N$ and weight $\geq 2$. Furthermore, the intersection of this space with $\Z[[q]]$ generates additively the the space of quasimodular forms of level $N$ and weight $k \geq 2$ with integral coefficients.
    \end{enumerate}
\end{theorem}

\begin{remark}
    In level one, the more natural traceforms come from the action of the symmetric group alone, which are known to be quasimodular by \cite{BachmannLectures}, and by \cite[Theorem 4.4]{CraigIttersumOno}) make up a spanning set for quasimodular forms of level one. Thus in this work, we will use the notation of \cite{CraigIttersumOno} when dealing in level one, noting that the difference between the two definitions is only a multiplicative constant (depending on the depth).
\end{remark}

The remainder of the paper is structured as follows. In Section \ref{S: Nuts and Bolts} we lay out many preliminary details required for the main proofs. In particular, we discuss the algebra of quasimodular forms for congruence subgroups and quasi--shuffle algebras. In Section \ref{S: q-MZV}, we give the definitions of the $q$-multiple divisor sums as laid down by Bachmann and K\"{u}hn in level one and Yuan and Zhao in level $N$, respectively. We describe the algebraic theory of these spaces, in particular the computation of the quasi-shuffle product, and we determine the intersection of these spaces with the algebra $\Z[[q]]$ of integer power series. In Section \ref{S: Main Proofs}, Theorems \ref{T: Main Theorem, Cyclotomic}, \ref{T: Main Theorem, Integral} and \ref{T: Main Theorem, Symmetry} are proven.

Further sections of this paper describe number-theoretic examples and applications of our main results, with special emphasis on linearization for quasimodular forms given in Theorem \ref{T: Main Theorem, Symmetry}; that is, we produce quasimodular forms without the need to multiply out Eisenstein series. In Section \ref{S: Partition cranks}, we show how the Atkin-Garvan crank moments can be written in terms of $q$-multiple zeta values. In Section \ref{S: Primes}, we describe results related to primes in arithmetic progressions along the lines of the recent papers of the author, van Ittersum and Ono \cite{CraigIttersumOno} and of Gomez \cite{Gomez}, with special emphasis on the role that quasimodularity plays in higher levels in prime-detecting partition functions. In Section \ref{S: Ramanujan Tau} we discuss a computation of the Ramanujan $\tau$-function with potential applications, and Section \ref{S: Quadratic forms} investigates quadratic forms in terms of $q$-multiple zeta values in levels 2 and 4. Finally, Section \ref{S: Questions} poses a number of open questions regarding applications of $q$-multiple zeta values to wider areas of the theories of partitions and modular forms. These sections contain several computational examples which may serve as exercises for the reader as well as inspiration for new applications.

\section*{Acknowledgements}

The author thanks Kathrin Bringmann, Jan-Willem van Ittersum and Ken Ono for helpful discussions which have improved this manuscript. The views expressed in this article are those of the author and do not reflect the official policy or position of the U.S. Naval Academy, Department of the Navy, the Department of Defense, or the U.S. Government.

\section{Nuts and Bolts} \label{S: Nuts and Bolts}

\subsection{Quasimodular forms for congruence subgroups}

In this section, we characterize quasimodular forms of level $N$ and provide the necessary preliminaries we will use to prove our main theorems. Let $\HH := \{ z \in \C : \mathrm{Im}(z) > 0 \}$ be the complex upper half-plane. Let $\Gamma$ denote any level $N$ congruence subgroup, where $N = 1$ or $N \geq 3$, so that $\begin{psmallmatrix} -1 & 0 \\ 0 & -1 \end{psmallmatrix} \not \in \Gamma$. These are the subgroups of $\Gamma(1) := \mathrm{SL}_2\lp \Z \rp$ containing the principal congruence subgroup $\Gamma(N)$ defined by
\begin{align*}
    \Gamma(N) := \left\{ \begin{pmatrix} a & b \\ c & d \end{pmatrix} \in \mathrm{SL}_2\lp \Z \rp : \begin{pmatrix} a & b \\ c & d \end{pmatrix} \equiv \begin{pmatrix} 1 & 0 \\ 0 & 1 \end{pmatrix} \bmod{N} \right\}.
\end{align*}
We will also have occasion to mention the particular congruence subgroups
\begin{align*}
    \Gamma_0(N) &:= \left\{ \begin{pmatrix} a & b \\ c & d \end{pmatrix} \in \mathrm{SL}_2\lp \Z \rp : \begin{pmatrix} a & b \\ c & d \end{pmatrix} \equiv \begin{pmatrix} * & * \\ 0 & * \end{pmatrix} \bmod{N} \right\}, \\
    \Gamma_1(N) &:= \left\{ \begin{pmatrix} a & b \\ c & d \end{pmatrix} \in \mathrm{SL}_2\lp \Z \rp : \begin{pmatrix} a & b \\ c & d \end{pmatrix} \equiv \begin{pmatrix} 1 & * \\ 0 & 1 \end{pmatrix} \bmod{N} \right\},
\end{align*}
where here $*$ denotes the permission of any congruence class modulo $N$.

For a positive integer $k \geq 1$, a congruence subgroup $\Gamma$, and a character $\chi$, a {\it modular form} of weight $k$ for $\Gamma$ with Nebentypus (or more generally, multiplier system) $\chi$ is a holomorphic function $f : \HH \to \C$ which is bounded near each rational number and, for $\gamma = \begin{psmallmatrix} a & b \\ c & d \end{psmallmatrix} \in \Gamma$, satisfies the functional equation
\begin{align*}
    f\lp \gamma z \rp := f\lp \dfrac{az + b}{cz + d} \rp = \chi(d) \lp cz + d \rp^k f(z).
\end{align*}
We say that a modular form has level $N$ if $f$ is modular with respect to a congruence subgroup of level $N$. Observe that modular forms of level $N$ are $N$-periodic since $\begin{psmallmatrix} 1 & N \\ 0 & 1 \end{psmallmatrix} \in \Gamma(N)$, and therefore modular forms have Fourier expansions. Thus, we define for any ring $R \subseteq \C$ the spaces $\cM_k\lp \Gamma, R, \chi \rp$ as the ring of modular forms of weight $k$ for $\Gamma$ with Fourier coefficients in $R$. If $\chi$ is trivial, we write $\cM_k\lp \Gamma, R \rp$, which in this paper will usually be the case. It is then easy to see that we have a graded algebra
\begin{align*}
    \cM\lp \Gamma, R, \chi \rp := \bigoplus_{k \geq 0} \cM_k\lp \Gamma, R, \chi \rp. 
\end{align*}
We note that $\cM_0\lp \Gamma(N), R \rp = R$. We also define the space of cusp forms $\mathcal{S}_k\lp \Gamma, R, \chi \rp$ to be the subspace of $\cM_k\lp \Gamma, R, \chi \rp$ having vanishing at all the cusps (i.e. $\chi(d)^{-1} (cz+d)^{-k} f(\gamma z)$ has vanishing constant term in its Fourier expansion for all $\gamma \in \mathrm{SL}_2\lp\Z \rp$, not just $\Gamma$), and we let the Eisenstein space $\mathcal{E}_k\lp \Gamma, R, \chi \rp$ be its orthogonal complement with respect to the Petersson inner product \cite{CohenStromberg}. We will in fact use the construction of the Eisenstein spaces as that space generated by all {\it Eisenstein series}, but the two formulations are equivalent.

One often considers modular forms for other congruence subgroups $\Gamma_0(N)$ and $\Gamma_1(N)$; we mostly ignore such considerations because of the well-known fact that traces yield surjective maps from $\Gamma(N)$-modular forms to modular forms of any level $N$-congruence subgroups. We also remind the reader that for a subgroup $\Gamma' \leq \Gamma$, we have $\cM_k\lp \Gamma, R, \chi \rp \subseteq \cM_k\lp \Gamma', R, \chi \rp$, and the trace operators (i.e. operators constructed by sums over cosets) give a surjective maps realizing this. This fact justifies the previous assertions that we need only work with $\Gamma(N)$, and indeed even that we can ignore finitely many cases of $\Gamma(N)$. We refer the reader to \cite{123,CohenStromberg,DiamondShurman,Ono,Shimura} for other useful facts about modular forms. In the remainder of this section, we describe the relevant parts of the theory of quasimodular forms, first in level 1 and subsequently in level $N$.

\subsubsection{Quasimodular forms in level 1}

We now formulate the basic elements of the theory of quasimodular forms for $\Gamma(1) = \mathrm{SL}_2\lp \Z \rp$; for a formal definition, see \cite{123}. We will use $\cQM_k\lp \Gamma, R \rp$ to denote the space of quasimodular forms of weight $k$ for $\Gamma$ with coefficients in $R$. The foundational quasimodular forms in level one are the Eisenstein series from the introduction, which we recall are defined for $k \geq 2$ even by
\begin{align*}
    E_k(z) = 1 - \dfrac{2k}{B_k} \sum_{n \geq 1} \sigma_{k-1}(n) q^n.
\end{align*}
It is well known that $E_k \in \cM\lp \Gamma(1), \C \rp$ for $k \geq 4$ and $E_2 \not \in \cM\lp \Gamma(1), \C \rp$ but $E_2 \in \cQM\lp \Gamma(1), \C \rp$.\footnote{It is also well-known that there are no quasimodular forms of weight 1 and level 1.} The central result describing quasimodular forms in level 1 is given below.

\begin{lemma} \label{L: Level 1 generation}
    We have $\cQM\lp \Gamma(1), \C \rp = \C[G_2, G_4, G_6]$.
\end{lemma}

We also require the action of the differential operator
\begin{align*}
    D := q \dfrac{d}{dq}
\end{align*}
on spaces of quasimodular forms. For $\ell \geq 0$, observe that $D^\ell : q^n \mapsto n^\ell q^n$. Ramanujan famously proved the differential identities
\begin{align*}
    DE_2 = \dfrac{E_2^2 - E_4}{12}, \ \ \ DE_4 = \dfrac{E_2 E_4 - E_6}{3}, \ \ \ DE_6 = \dfrac{E_2 E_6 - E_4^2}{2}.
\end{align*}
These identities along with Proposition \ref{L: Level 1 generation} yield the following closure property.

\begin{lemma} \label{L: Derivative Level 1}
    The algebras $\cQM\lp \Gamma(1), \C \rp$ and $\cQM\lp \Gamma(1), \Q \rp$ are closed under the action of $D$.
\end{lemma}

This gives us a basic description of the space of quasimodular forms of level 1 from which we now build intuition for the further theory. We note for instance \cite{AtkinGarvan,KanekoZagier} for some of the earliest systematic work on quasimodular forms.

\subsubsection{Quasimodular forms in level $N$}

We now describe how the theory of quasimodular forms extends to level $N$ for $N \geq 3$. We first make use of the following important result, establishing $E_2$ as fundamental.

\begin{proposition}[{\cite[Proposition 20]{123}}] \label{P: Polynomial in E_2}
    Let $\Gamma$ be a congruence subgroup. Then $\cQM\lp \Gamma, \C \rp = \cM\lp \Gamma, \C \rp[E_2]$.
\end{proposition}

It is also useful to know that the generation of algebras of quasimodular forms does not have any nontrivial dependence on the field of coefficients, which is encapsulated in the result below.

\begin{lemma} \label{L: Sturm idea}
    Let $\Gamma$ be a congruence subgroup. Then if the algebra $\cQM\lp \Gamma, \Q \rp$ is generated as an algebra by a subset with maximal weight $k$, then the same is true for $\cQM\lp \Gamma, K \rp$ for any field $\Q \subseteq K \subseteq \C$.
\end{lemma}

\begin{proof}
    See for example \cite[Theorem 10.12]{Shimura} or \cite[Lemma 3]{Rustom}.
\end{proof}

Because of Lemma \ref{L: Sturm idea}, we need not be very concerned with the field of definition for the Fourier coefficients of modular forms we consider. In particular, this lemma will render Theorem \ref{T: Main Theorem, Cyclotomic} and Theorem \ref{T: Main Theorem, Integral} equivalent.

By Proposition \ref{P: Polynomial in E_2} and by taking traces, we can now describe $\cQM\lp \Gamma, \C \rp$ for any congruence subgroup $\Gamma$ of level $N$ by describing the space $\cM\lp \Gamma(N), \C \rp$. While results as strong and straightforward as Lemma \ref{L: Level 1 generation} do not exist in general level, we will summarize two methods of ``almost" generating $\cM\lp \Gamma(N), \C \rp$; in particular, we will only miss out on the subspace $\cS_1\lp \Gamma(N), \C \rp$.

\subsubsection{Basics of Eisenstein series of level $N$}

The main results we will require for generating $\cQM\lp \Gamma, \C \rp$ require facts about Eisenstein series for level $N \geq 3$.\footnote{Recall that generating $\cQM\lp \Gamma(4), \C \rp$ is enough to generate $\cQM\lp \Gamma(2), \C \rp$ by taking traces.} Let $\lambda = \lp a,b \rp \in  \Z^2/N\Z^2$; we then define a two-variable Eisenstein series of weight $\ell \geq 1$ by
\begin{align} \label{Eq: Two-Var Eisenstein series}
    E_{\ell,\lambda, N}(z,s) = \sum_{\substack{(m,n) \equiv \lambda \bmod{\Z^2/(N\Z)^2} \\ (m,n) \not = (0,0)}} \dfrac{y^{s/2}}{\lp mz + n \rp^\ell \left| mz + n \right|^s}.
\end{align}
As usual we have $z = z_1 + z_2 i \in \HH$. The series \eqref{Eq: Two-Var Eisenstein series} converges for $\mathrm{Re}(s) + \ell > 2$ by the classical Hecke trick, and continues meromorphically to all $s \in \C$ such that\footnote{Here, $\Gamma(s)$ is the classical gamma function. Although we use $\Gamma$ to denote both the gamma function and congruence subgroups, it is always clear from context which is intended.} $\Gamma\lp \lp \frac s2 \rp + \ell \rp E_{\ell,\lambda,N}(z,s)$ is entire (see \cite[Theorem 9.7]{Shimura}). With this analytic continuation, we define
\begin{align} \label{Eq: Hecke trick}
    E_{\ell,\lambda,N}(z) := E_{\ell,\lambda,N}(z,0).
\end{align}
The function \eqref{Eq: Hecke trick} is holomorphic for all $\ell \geq 1$ except for $\ell = 2$, in which case $E_{2,\lambda,N}(z) + \frac{\pi}{z_2}$ is holomorphic. Following standard arguments, each $E_{\ell,\lambda,N}(z)$, $\ell \not = 2$ is a holomorphic modular form for $\Gamma(N)$ and $E_{2,\lambda,N}(z)$ is a quasimodular form for $\Gamma(N)$. These Eisenstein series span the full space of Eisenstein series for $\Gamma(N)$ \cite[Ch. 7]{Miyake}. Apart from the modularity of Eisenstein series, we will need the Fourier expansions of these series at the cusp at $i\infty$. The lemma below gives these expansions.

\begin{proposition}[{\cite[Equations 9.4--9.5]{Shimura}}] \label{P: Expansion of Eisenstein series}
    Let $k \geq 1$, $N \geq 3$ be integers and $\lambda = (a,b) \in \lp \Z/N\Z \rp^2$. Then we have
    \begin{align*}
        E_{k,\lambda,N}(z) = \dfrac{(-1)^k}{\lp 2\pi i \rp^k N^k \Gamma(k)} \bigg[ A + B &+ \sum_{\ell > 0} \sum_{\substack{m>0 \\ m \equiv a \bmod{N}}} \ell^{k-1} \zeta_N^{b\ell} q^{m\ell/N} \\ &+ \sum_{\ell > 0} \sum_{\substack{m>0 \\ m \equiv -a \bmod{N}}} (-1)^k \ell^{k-1} \zeta_N^{-b\ell} q^{m\ell/N} \bigg],
    \end{align*}
    where we define
    \begin{align*}
        A := \begin{cases}
            \lp -2\pi i \rp^{-k} N^k \Gamma(k) Z_{b,N}^k(0) & \text{if } a = 0, \\ 0 & \text{otherwise},
        \end{cases}
    \end{align*}
    and
    \begin{align*}
        B := \begin{cases}
            2^{-1} Z_{b,N}^1\lp - \dfrac 12 \rp & \text{if } k=1 \text{ and } a \not = 0, \\ \lp 4\pi y \rp^{-1} & \text{if } k=2, \\ 0 & \text{otherwise},
        \end{cases}
    \end{align*}
    where
    \begin{align*}
        Z_{b,N}^k(s) := \sum_{0 \not = m \equiv b \bmod{N}} \dfrac{1}{m^k |m|^{2s}}.
    \end{align*}
\end{proposition}

These Fourier expansions will give a direct link between Eisenstein series of level $N$ and $q$-multiple zeta values of level $N$.

\subsubsection{Generation by Eisenstein series of level $N$}

We now give two results from the literature regarding the generation of the algebras $\cM\lp \Gamma(N), \C \rp$. We first state the main result of \cite{Khuri-Makdisi}, which has a strong focus on weight one.

\begin{theorem}[{\cite[Theorem 5.1]{Khuri-Makdisi}}] \label{T: T: Eisenstein Generation Weight 1}
    Let $N \geq 3$. Then the algebra generated by the Eisenstein $E_{1,\lambda,N}(z)$ as $\lambda = (a,b)$ ranges over $\lp \Z/ N\Z \rp^2$, contains all modular forms of weights $\geq 2$ for $\Gamma(N)$.
\end{theorem}

We can then obtain immediately as a corollary the following result on the generation of $\cQM\lp \Gamma, \C \rp$.

\begin{corollary} \label{C: Construction 1}
    Let $N \geq 2$ be an integer. Then $\cQM\lp \Gamma, \C \rp$ is generated as an algebra by the Eisenstein series $E_{1,\lambda,N}(z)$, $E_2(z)$, and any basis for the space $\cS_1\lp \Gamma, \C \rp$.
\end{corollary}

\begin{proof}
    By taking traces from $\Gamma(N)$ to $\Gamma$ and from $\Gamma(4)$ to $\Gamma(2)$, we need only prove the result for $\Gamma(N)$ for $N \geq 3$. The remainder follows from Proposition \ref{P: Polynomial in E_2} and Theorem \ref{T: T: Eisenstein Generation Weight 1}.
\end{proof}

Thus, the Eisenstein series of weights $\leq 2$ are the foundational objects for describing modular forms in level $N \geq 2$. However, in order to obtain results in larger weights, this would require taking the product of many Eisenstein series. For future applications, it may be convenient to allow Eisenstein series of larger weights to enter into calculations. To this end, we state a result of Raum and Xia \cite{RaumXia}. 

\begin{theorem}[{\cite[Theorem 1]{RaumXia}}] \label{T: Two Eisenstein product}
    Let $k,\ell$, and $N$ be positive integers. Then there is a positive integer $N_0$ such that
    \begin{align*}
        \cM_{k+\ell}\lp \Gamma(N), \C \rp \subseteq \cE_{k+\ell}\lp \Gamma(N) \rp + \cE_k\lp \Gamma(N_0) \rp \cdot \cE_\ell\lp \Gamma(N_0) \rp.
    \end{align*}
    Furthermore, if $k+\ell \geq 3$ then $N_0$ is explicitly computable, and there is an explicitly computable $N_1$ such that
    \begin{align*}
        \cM_{k+\ell}\lp \Gamma(N), \C \rp \subseteq \cE_k\lp \lcm\lp N_0, NN_1 \rp \rp \cdot \cE_\ell\lp \lcm\lp N_0, N_1 \rp \rp.
    \end{align*}
\end{theorem}

\begin{remark}
    We leave make the following remarks on Theorem \ref{T: Two Eisenstein product}:
    \begin{enumerate}
        \item The value of $N_0$ is computed in \cite[Theorem 4.4]{RaumXia} in terms of the Sturm bound on $\cM_{k+\ell}\lp \Gamma_0(N) \rp$, and $N_1$ can be computed from \cite[Theorem 5.2]{RaumXia}. To be precise, $N_0$ is, for $\Gamma_0(N)$-modular forms, $N$ times the Sturm bound for $\cM_{k+\ell}\lp \Gamma_0(N) \rp$.
        \item The benefit of Theorem \ref{T: Two Eisenstein product} from a computational perspective is that products of at most two Eisenstein series appear, which reduces the number of terms appearing in the representation of a given modular form by allowing the level to vary after quasi-shuffle product machinery is applied. The cost is that one must work in larger spaces of modular forms.
    \end{enumerate}
\end{remark}

Thus, if we allow Eisenstein series in larger weights to enter into our consideration, we can represent modular forms of larger weights using relatively simple products. We may thus derive the following corollary immediately.

\begin{corollary} \label{C: Construction 2}
    Each element of $\cQM\lp \Gamma, \C \rp$ of weight $\geq 2$ can be expressed as a polynomial in $E_2$ whose coefficients are linear combinations of products of at most two Eisenstein series (of level possibly larger than $N$) or products of Eisenstein series of weight one and level $N$.
\end{corollary}

We will close this section with a brief discussion of the relationship between these Eisenstein series, whose coefficients naturally live in $\Q\lp\zeta_N\rp$, and quasimodular forms of level $N$ with coefficients in $\Z[[q]]$.

\begin{lemma} \label{L: Integral Eisenstein combinations}
    Suppose a quasimodular form $f$ of weight $k$ and level $N$ has integral Fourier coefficients apart from possibly the constant term. Then such forms are linear combinations of the traceforms
    \begin{align*}
        E_{k,(a,b),N}^{\mathrm{int}} := \sum_{\sigma \in \mathrm{Gal}\lp \Q\lp \zeta_N \rp/\Q \rp} \sigma E_{k,(a,b),N},
    \end{align*}
    where as usual $\mathrm{Gal}(K/\Q)$ denotes the Galois group of the field extension $K/\Q$.
\end{lemma}

\begin{proof}
    Assume first the weight of $f$ is purely $k$. We observe and use throughout that the coefficients of each $E_{k,(a,b),N}$ live in the field $\Q\lp \zeta_N \rp$. Suppose a quasimodular form $f$ has integral coefficients. Then any element $\sigma \in \mathrm{Gal}\lp \Q\lp \zeta_N \rp / \Q \rp$ fixes the coefficients of $f$; thus, if $E_{k,(a,b),N}$ appears in the formula for $f$, then for each $x$ coprime to $N$, the form $\sigma_x E_{k,(a,b),N} = E_{k,(a,bx),N}$ appears in $f$ with the same coefficient. Thus, $f$ is built from the traceforms. Likewise, each traceform is fixed by each element of $\mathrm{Gal}\lp \Q\lp \zeta_N \rp/\Q \rp$ and thus the coefficients live in $\Q$, but since the coefficients are also evidently elements of $\Z[\zeta_N]$, the coefficients live in $\Q \cap \Z[\zeta_N] = \Z$. This completes the proof for pure weights, and the proof extends to mixed weights since the algebra of quasimodular forms is graded by weights.
\end{proof}

Because of Lemma \ref{L: Sturm idea}, we also have the following corollary.

\begin{corollary} \label{C: Integer weight 2 generation}
    The quasimodular forms of level $N$ with integer coefficients and weights $k \geq 2$ are generated as an algebra by the forms from Lemma \ref{L: Integral Eisenstein combinations} and $E_2$.
\end{corollary}

\subsection{Quasi-shuffle algebras}

In order to work with $q$-multiple zeta values of any level, we need to clarify the important machinery of quasi-shuffle algebras (see \cite{Hoffman00,HoffmanIhara,HoffmanWeb,IKOO}). In order to define these, we must fix some notation. Suppose we have a set $A$, which we shall call the {\it set of letters}. Define on $A$ a commutative, associative product $\diamond$. We may extend this product bilinearly to form a commutative non-unital algebra $\lp \Q A, \diamond \rp$. We are concerned with $\Q$-linear combinations of concatenated letters (i.e. elements) of $A$, a space we denote by $\Q\langle A \rangle$. Monomials in $\Q\langle A \rangle$ are called {\it words}, and the concatenation of two words $w,x$ will be denoted $wx$.

\begin{definition}
    Let $A, \diamond$ be constructed as above. The {\it quasi-shuffle product} $*_\diamond$ on $A$ is defined as the unique $\Q$-bilinear product on $A$ satisfying $1 *_\diamond w = w *_\diamond 1 = w$, for any letters $a,b \in A$ and words $w,t \in \Q\langle A \rangle$, the recursive rule
    \begin{align*}
        aw *_\diamond bt = a\lp w *_\diamond bt \rp + b\lp aw *_\diamond t \rp + \lp a \diamond b \rp \lp w *_\diamond t \rp.
    \end{align*}
\end{definition}

The central example of quasi-shuffle products come from the theory of multiple zeta values as describe in the introduction. To see the connection between multiple zeta values and the quasi-shuffle product, we consider the case $A = \N$ and $a \diamond b = a+b$ for $a,b \in A$. In this light, the concatenation $ab$ of two natural numbers is really the tuple $\lp a,b \rp$, and likewise concatenation of a letter $a$ to any word $w$ is just the extension of that tuple $w$ by $a$ in the first position. Observe then that we can define the multiple zeta function $\zeta$ on monomials in $\Q\langle A \rangle$ in the natural way, provided that word contains a letter different from 1. Then the famous relation
\begin{align*}
    \zeta\lp k_1 \rp \zeta\lp k_2 \rp = \zeta\lp k_1, k_2 \rp + \zeta\lp k_2, k_1 \rp + \zeta\lp k_1 + k_2 \rp
\end{align*}
can be interpreted in terms of the quasi-shuffle products; more broadly, the fully general shuffle relations for multiple zeta values can be expressed by $\zeta\lp w * v \rp = \zeta(w) \zeta(v)$; proofs of this fact and of similar identities abound in the literature on multiple zeta values and quasi--shuffle algebras. Thus, the quasi-shuffle structure of $\mathcal{MZ}$ reduces product computations to sums.

In Sections \ref{S: q-MZV Level One} and \ref{S: q-MZV Level N}, we will see how this structure lifts to $q$-multiple zeta values of levels 1 and $N$, respectively. The particular focus will be, just as in the case of multiple zeta values, that products of $q$-multiple zeta values can be interpreted as linear combinations of $q$-multiple zeta values by means of a quasi-shuffle product.

There are many nice results for studying quasimodular $q$-multiple zeta values that arise from the literature on quasimodular forms. For instance, in \cite{BCIP} a certain quasi-shuffle exponential equation \cite[Equation 32]{HoffmanIhara} is used to generate infinite family of quasimodular forms from low-weight examples, and in \cite{CraigIttersumOno} the authors use a formula for the trace of a natural action of $S_n$, the $n$th symmetric group action, on quasi-shuffle algebras. Since we use a similar style of argument in the course of Theorem \ref{T: Main Theorem, Symmetry}, we cite this result, although we do not use it directly.

\begin{proposition}[{\cite[Theorem 13]{Hoffman20}}] \label{P: Quasishuffle symmetry}
    For $a_1, \dots, a_n \in A$, we have
    \begin{align*}
        \sum_{\sigma \in S_n} a_{\sigma(1)} \cdots a_{\sigma(n)} = \sum_{B \in \Pi_n} f(B) \prod_{\beta \in B} \diamond_{j \in \beta} a_j,
    \end{align*}
    where $\prod$ here denotes a product with respect to $*$, $\Pi_n$ denotes all set partitions of $\{ 1, 2, \dots, n \}$, $\diamond_{j \in B}$ denotes a $\diamond$-product over all elements of $B$, and
    \begin{align*}
        f(B) = (-1)^{n-|B|} \prod_{\beta \in B} \lp |\beta| - 1 \rp!.
    \end{align*}
\end{proposition}
In particular, we note that in \cite{CraigIttersumOno} this formula is used to show that symmetric sums of $q$-multiple zeta values of level one produce quasimodular forms by showing that the coefficients arising from the $\diamond$-product filter out odd-index Eisenstein series.

\section{$q$-multiple zeta values} \label{S: q-MZV}

In this section, we describe central results in the theories of $q$-multiple divisor sums and $q$-multiple zeta values. We begin with the level one theory due to Bachmann and K\"{u}hn \cite{BachmannKuhn}, and we follow up with the generalization to level $N$ due to Yuan and Zhao \cite{YuanZhao}. In the process, we describe (in level one and level $N$) the relationship between $q$-multiple zeta values and $q$-multiple divisor sums, quasi-shuffle algebra structure of $q$-multiple divisor sums, and connections between $q$-multiple zeta functions and Eisenstein series of level $N$.

\subsection{Level one $q$-multiple zeta values} \label{S: q-MZV Level One}

In this section, we describe the central aspects of the $q$-multiple zeta value theory of Bachmann and K\"{u}hn \cite{BachmannKuhn}. We define for $a \in \Z_{\geq 0}$, $\vec{k} = \lp k_1, \dots, k_a \rp \in \Z_{\geq 0}^a$, $\vec{P} = \lp P_1, \dots, P_a \rp \in \lp \Q[x] \rp^a$ with $P_1 \in x \Q[x]$ and each $P_j$ of degree at most $k_j$, the functions
\begin{align*}
    \mathcal{U}_{\vec{k}; \vec{P}}(q) := \sum_{0 < n_1 < n_2 < \dots < n_a} \prod_{j=1}^a \dfrac{P_j(q^{n_j})}{\lp 1 - q^{n_j} \rp^{k_j}}.
\end{align*}
These are examples of $q$-analogs of multiple zeta values as defined by many authors, which is justified by the relation \cite[p. 6]{BachmannKuhn}
\begin{align*}
    \lim_{q \to 1} \lp 1 - q \rp^{k_1 + \dots + k_a} \mathcal{U}_{(k_1, \dots, k_a); (P_1, \dots, P_a)}(q) = \zeta\lp k_1, \dots, k_a \rp \prod_j P_j(1).
\end{align*}
See Brindle's overview \cite{Brindle} for a discussion of various different constructions of $q$-multiple zeta values. We denote by $\mathcal{MZ}_q$ the algebra generated by all $q$-multiple zeta values in the above sense. We note that under the interpretation as sums over partitions into exactly $a$ distinct part sizes, the constant functions belong to $\mathcal{MZ}_q$ from the case $a=0$.

We consider as a special case of these $q$-multiple zeta values
\begin{align*}
    \mathcal{U}_{\vec{k}}(q) := \sum_{0 < n_1 < n_2 < \dots < n_a} \dfrac{q^{n_1 + n_2 + \dots + n_a} Q_{k_1}(q^{n_1}) Q_{k_2}(q^{n_2}) \cdots Q_{k_a}(q^{n_a})}{\lp 1 - q^{n_1} \rp^{k_1+1} \lp 1 - q^{n_2} \rp^{k_2+1} \cdots \lp 1 - q^{n_a} \rp^{k_a+1}}.
\end{align*}
where $Q_k(x)$ are the {\it Eulerian polynomials} defined by the relation
\begin{align*}
    \sum_{n \geq 1} n^k x^n = \dfrac{x Q_k(x)}{\lp 1 - x \rp^{k+1}}.
\end{align*}

We cite here the central results of Bachmann and K\"{u}hn. The first result concerns the space $\mathcal{MD}_q$, which we define a priori as the $\Q$-vector space generated by all forms $\mathcal{U}_{\vec{k}}(q)$ for $\vec{k}$ a vector of any length of non-negative integers. The main results of \cite{BachmannKuhn} with which we are concerned are the following.

\begin{proposition}[{\cite[Lemma 2.5]{BachmannKuhn}}] \label{P: BachmannKuhn Series}
    For $\vec{k} = \lp k_1, \dots, k_a \rp \in \Z_{\geq 0}^a$, we have
    \begin{align*}
        \mathcal{U}_{\vec{k}}(q) = \sum_{n \geq 0} \sigma_{\vec{k}}(n) q^n = \sum_{\substack{m_1 n_1 + \dots + m_a n_a = n \\ 0 < n_1 < n_2 < \dots < n_a}} m_1^{k_1} m_2^{k_2} \cdots m_a^{k_a}.
    \end{align*}
\end{proposition}

\begin{proposition}[{\cite[Proposition 4.1 (i)]{BachmannKuhn}}] \label{P: Quasimodular level one surjection}
    All quasimodular forms of level one are represented by elements of the algebra $\mathcal{MD}_q$.
\end{proposition}

\begin{theorem}[{\cite[Theorems 1.3 and 1.7]{BachmannKuhn}}] \label{T: BachmannKuhn Closure}
    $\mathcal{MD}_q$ is a differential subalgebra of $\mathcal{MZ}_q$ with the differential operator $D$. That is, $\mathcal{MD}_q$ is closed under multiplication and under action by $D$.
\end{theorem}

We note that Proposition \ref{P: BachmannKuhn Series} gives the coefficients of $\mathcal{U}_{\vec{k}}(q)$ as the case $N=1$ of twisted multiple divisor sums as in \eqref{Eq: Twisted multiple divisor sum}, that is, the work of Bachmann and K\"{u}hn is a genuine level 1 case of the broader theory of Yuan and Zhao which we discuss next. We also note that Proposition \ref{P: BachmannKuhn Series} represents $q$-multiple divisor sums as sums over partitions into $a$ distinct part sizes; thus, we see the interface between $q$-multiple zeta values and partitions. The proof of Theorem \ref{T: BachmannKuhn Closure} uses crucially the quasi-shuffle algebra structure of $\mathcal{MD}_q$ as well as the fact that each Eisenstein series $E_{2k}$ is a clear example of a depth one $q$-multiple divisor sum of weight $2k$. An explicit formula for products in level one is given, for example, in \cite[Proposition 4.3]{CraigIttersumOno}, but will also be discussed in a later section of this paper.

\subsection{Level $N$ $q$-multiple zeta values} \label{S: q-MZV Level N}

In this section, we describe the central aspects of the $q$-multiple zeta value theory of Yuan and Zhao \cite{YuanZhao}, which generalizes the theory of Section \ref{S: q-MZV Level One}. For $a \in \Z_{\geq 0}$, $N \geq 1$, $\vec{k} = \lp k_1, \dots, k_a \rp \in \Z_{\geq 0}^a$, and $\vec{c} = \lp c_1, \dots, c_a \rp \in \lp \Z/N\Z \rp^a$, we define (as in the introduction) using the twisted multiple divisor sums \eqref{Eq: Twisted multiple divisor sum} the series\footnote{When the level is clear from context, the $N$ is omitted from the notation for convenience. In the notation of Yuan and Zhao, $[\vec{k} + \vec{1},\vec{c}] = \lp \prod_j k_j! \rp^{-1} \mathcal{U}_{\vec{k},\vec{c}}(q)$, where $\vec{1} = \lp 1, 1, \dots, 1 \rp$. Both notations are useful from different points of view.}
\begin{align*}
    \mathcal{U}_{\vec{k};\vec{c};N}(q) = \mathcal{U}_{\vec{k};\vec{c}}(q) := \sum_{n > 0} \sigma_{\vec{k};\vec{c}}(n) q^n.
\end{align*}
We suppress $N$ from the notation where there is no risk of confusion. Let $\mathcal{MD}_q(N)$ denote the $\Q\lp \zeta_N \rp$-vector space generated by all possible forms $\mathcal{U}_{\vec{k},\vec{c}}(q)$. This is a subspace of the space $\mathcal{MZ}_q(N)$ of all $q$-multiple zeta values of level $N$, which are series of the form
\begin{align*}
    \sum_{0 < n_1 < n_2 < \dots < n_a} \prod_{j=1}^a \dfrac{P_j\lp \zeta_N^{c_j} q^{n_j} \rp}{\lp 1 - \zeta_N^{c_j} q^{n_j} \rp^{k_j+1}} \in \Q\lp \zeta_N \rp[[q]],
\end{align*}
where $P_j(x) \in x\Q[x]$ are polynomials. For such forms, we define the {\it weight} to be $|\vec{k}| + a$ and the {\it depth} to be $a$. The first main stage of generalizing the work of Bachmann and K\"{u}hn rests in generalizing Proposition \ref{P: BachmannKuhn Series}. This is accomplished in the following result.

\begin{proposition}[{\cite[Lemma 6.10]{YuanZhao}}]
    We have for $a,N \geq 1$, $\vec{k} \in \N^a$, $\vec{c} \in \lp \Z/N\Z \rp^a$ the $q$-series representation
    \begin{align*}
        \mathcal{U}_{\vec{k};\vec{c}}(q) = \sum_{0 < n_1 < n_2 < \dots < n_a} \dfrac{\zeta_N^{c_1 + c_2 + \dots + c_a} q^{n_1 + n_2 + \dots + n_a} Q_{k_1}\lp \zeta_N^{c_1} q^{n_1} \rp Q_{k_2}\lp \zeta_N^{c_2} q^{n_2} \rp \cdots Q_{k_a}\lp \zeta_N^{c_a} q^{n_a} \rp}{\lp 1 - \zeta_N^{c_1} q^{n_1} \rp^{k_1+1} \lp 1 - \zeta_N^{c_2} q^{n_2} \rp^{k_2 + 1} \cdots \lp 1 - \zeta_N^{c_a} q^{n_a} \rp^{k_a + 1}}
    \end{align*}
    where $Q_k(x)$ are the Eulerian polynomials.
\end{proposition}

Observe that $\mathcal{U}_{\vec{k};\vec{c}}(q)$ is thus a natural $\zeta_N$-twist of $\mathcal{U}_{\vec{k}}(q)$. To connect $q$-multiple divisor sums at level $N$ with partitions, it suffices to consider the original definition of the twisted multiple divisor sum
\begin{align*}
    \sigma_{\vec{k},\vec{c}}(n) := \sum_{\substack{m_1 n_1 + \dots + m_a n_a = n \\ 0 < n_1 < n_2 < \dots < n_a}} \zeta_N^{c_1 m_1 + c_2 m_2 + \dots + c_a m_a} m_1^{k_1} m_2^{k_2} \cdots m_a^{k_a},
\end{align*}
which is by construction a sum over partitions into $a$ distinct part sizes, and is again a natural $\zeta_N$-twist of the multiple divisor sums considered by Bachmann and K\"{u}hn.

We also consider now the level $N$ generalization of Theorem \ref{T: BachmannKuhn Closure}.

\begin{theorem} \label{T: YuanZhao Closure}
    For any $N \geq 1$, we have that $\mathcal{MD}_q(N)$ is a differential subalgebra of $\mathcal{MZ}_q(N)$.
\end{theorem}

These results summarize the main ideas of \cite{YuanZhao} generalizing those in \cite{BachmannKuhn}. The main theorem of this paper completes the work of \cite{YuanZhao} from the number theorists' point of view by proving Theorems \ref{T: Main Theorem, Cyclotomic} and \ref{T: Main Theorem, Integral}, which are analogs of Proposition \ref{P: Quasimodular level one surjection} for level $N$. See also \cite{Kanno} for discussions of multiple Eisenstein series and $q$-multiple zeta values at level $N$. We will discuss the quasi-shuffle structure after an aside on $q$-multiple zeta values with integer coefficients.

\subsection{$q$-multiple zeta values with integral coefficients}

We now discuss certain $q$-multiple zeta values of level $N$ living in $\Z[[q]]$ rather than $\Q\lp \zeta_N \rp[[q]]$. We note that forms in level one already reside in $\Z[[q]]$, and so we focus only on $N>1$. We firstly note the trivial case of $\mathcal{U}_{\vec{k};\vec{0}}(q)$, which reduces back to $q$-multiple zeta values of level one. We thus restrict our focus to cases which do not reduce to level one.

We construct these forms along the lines proposed in part by MacMahon in \cite{MacMahon} and more fully by Rose \cite{Rose} and Larson \cite{Larson}. For construction, we observe the utility of the orthogonality of $N$th roots of unity. In particular, it is plain to see that we have for each $0 \leq d < N$ that the sums
\begin{align*}
    \sum_{c=0}^{N-1} \zeta_N^{-cd} \sigma_{\vec{k},(c_1,\dots,c,\dots,c_a)}(n)
\end{align*}
modify the multiple divisor sums by imposing the restriction $m_j \equiv d \pmod{N}$. Observe that such modifications occur purely in $\mathcal{MD}_q(N)$ without twists on $q$. Thus, by taking these $\Q\lp\zeta_N\rp$-linear combinations of certain forms $\mathcal{U}_{\vec{k};\vec{c}}(q)$ we can impose any desired congruence conditions on the multiplicities $m_j$. Now, recalling $m_j$ is the multiplicity of $n_j$ in the $q$-series $\mathcal{U}_{\vec{k};\vec{c}}(q)$, the restriction $m_j \equiv d \pmod{N}$ imposes the change of variables $n_j \mapsto N n_j + d$, i.e., we may represent these $q$-series using a sum $n_j \equiv d \pmod{N}$. Thus, if $S_j$ denote any subsets of $\{ 0, 1, \dots, N-1 \}$ and if $\vec{S} = \lp S_1, \dots, S_a \rp$ denotes the collection of these residue classes, we have that\footnote{We abuse notation here, but the difference is clear between a vector of subsets of integers and a vector of polynomials.}
\begin{align} \label{Eq: Integral forms}
    \mathcal{U}_{\vec{k};\vec{S}}(q) = \sum_{\substack{0 < n_1 < n_2 < \dots < n_a \\ n_j \in S_j \bmod{N}}} \dfrac{q^{n_1 + \dots + n_a} Q_{k_1}\lp q^{n_1} \rp \cdots Q_{k_a}\lp q^{n_a} \rp}{\lp 1 - q^{n_1} \rp^{k_1+1} \cdots \lp 1 - q^{n_a} \rp^{k_a+1}} =: \sum_{n>0} \sigma_{\vec{k},\vec{S}}(n) q^n \in \mathcal{MD}_q(N),
\end{align}
and we see that
\begin{align} \label{Eq: Divisors in progressions}
    \sigma_{\vec{k},\vec{S}}(n) = \sum_{\substack{m_1 n_1 + \dots + m_a n_a = n \\ 0 < n_1 < n_2 < \dots < n_a \\ m_j \in S_j \bmod{N}}} m_1^{k_1} \cdots m_a^{k_a}.
\end{align}
By following the same reasoning, we have the following result.

\begin{proposition}
    For $N \geq 1$ and $S_j \subset \{ 0, 1, \dots, N-1 \}$, and for any $a \geq 1$, $\vec{k} = \lp k_1, \dots, k_a \rp \in \Z_{\geq 0}^a$ and $P_1(x), \dots, P_a(x) \in \Q[x]^a$, we have
    \begin{align*}
        \sum_{\substack{0 < n_1 < n_2 < \dots < n_a \\ n_j \in S_j \bmod{N}}} \prod_{j=1} \dfrac{q^{n_1 + \dots + n_a} P_{k_1}(q^{n_1}) \cdots P_{k_a}(q^{n_a})}{\lp 1 - q^{n_1} \rp^{k_1+1} \cdots \lp 1 - q^{n_a} \rp^{k_a+1}} \in \mathcal{MZ}_q(N) \cap \Z[[q]].
    \end{align*}
\end{proposition}

We note that once Theorem \ref{T: Main Theorem, Integral} is demonstrated in Section \ref{S: Main Proofs}, we obtain as a corollary that all quasimodular forms of level $N$ with integral coefficients are generated by $E_2$ along with $\mathcal{E}_1\lp \Gamma(N) \rp \cap \Z[[q]]$. We know by Theorem \ref{T: Main Theorem, Cyclotomic} that if we permit $\zeta_N$-twists of elements of $\mathcal{MD}_q(N)$ then we can obtain the full space $\mathcal{E}_1\lp \Gamma(N) \rp$. Rose \cite{Rose} and Larson \cite{Larson} have understood much about the quasimodularity of these series, and in particular they demonstrate quasimodularity in the case of $\mathcal{U}_{(1,\dots,1);(S_1,\dots,S_a)}(q)$, although Lemma \ref{L: Sturm idea} and Galois traces will be the true devices behind establishing a fuller understanding of $\Z[[q]]$-valued $q$-multiple zeta values.

\subsection{Quasi-shuffle algebras and $q$-multiple divisor sums}

The contents of the previous sections describe the main results of the theory of $q$-multiple divisor sums from an algebraic point of view. We now dive into the computational details which explain the proof of Theorem \ref{T: YuanZhao Closure} by means of exact formulas. The machinery here is that of quasi-shuffle algebras. We will begin by describing the algorithmic process whereby quasi-shuffle algebras produce a formula for products of $q$-multiple divisor sums.

To describe the necessary quasi-shuffle algebras let $\mathbb{F}$ be either $\Q$ or $\Q_N$, and let
\begin{align*}
    A = \{ z_{j,c} : j \in \N, c \in \Z/N\Z \}
\end{align*}
be the desired alphabet. For any commutative and associative product $\diamond$ on $\mathbb{F} A$, we obtain a quasi-shuffle product $* = *_\diamond$ which is defined recursively by $1*w = w*1 = w$ for each $w \in \mathbb{F}\langle A \rangle$ and, for all words $w,v \in \mathbb{F}\langle A \rangle$ and all letters $a,b \in A$,
\begin{align*}
    aw * bv = a\lp w*bv \rp + b\lp aw * v \rp + \lp a \diamond b \rp \lp w*v \rp.
\end{align*}
In \cite[Lemma 5.3]{YuanZhao}, it is shown that certain normalizations of the polylogarithms $\mathrm{Li}_{-s}(z) = \frac{z Q_s(z)}{(1 - z)^{s+1}}$ satisfy a relation that fits very cleanly into this quasi-shuffle framework. Motivated by that formula, Yuan and Zhao define the coefficients $\omega_{n;c}^N$ and $\lambda_{a,b;c}^{j;N}$ for $a,b,j,N \geq 1$ and $c \in \Z/N\Z$ by
\begin{align*}
    \dfrac{1}{\zeta_N^c e^x - 1} =: \dfrac{\delta_{c,0}}{x} + \sum_{n \geq 0} \dfrac{\omega_{n;c}^N}{n!} x^n, \ \ \ \lambda_{a,b;c}^{j;N} := (-1)^{b-1} \binom{a+b-j-1}{a-j} \dfrac{\omega_{a+b-j-1;c}^N}{\lp a+b-j-1 \rp!}
\end{align*}
and define the $\diamond$-product on $A$, for $c,d \in \Z/N\Z$ and $a,b \in \N$,
\begin{align*}
    z_{a;c} \diamond z_{b;d} = \sum_{j=1}^a \lambda_{a,b;c-d}^{j;N} z_{j;c} + \sum_{j=1}^b \lambda_{b,a;d-c}^{j;N} z_{j;d} + \delta_{c,d} z_{a+b;c}.
\end{align*}
Because this product arises by identifying $z_{a;c} \mapsto \widetilde{\mathrm{Li}}_s\lp \zeta_N^c q^a \rp$ for the normalized polylogarithms used by Yuan and Zhao, it is automatically commutative and associative, and thus gives rise to a quasi-shuffle product. This construction then gives rise to the calculation of products of $q$-multiple divisor sums. To define these, we will say vectors $\vec{a}, \vec{b}$ of not necessarily the same dimension, and for constants $c_{\vec{a}}, c_{\vec{b}}$ that
\begin{align*}
    \mathcal{U}_{c_{\vec{a}} \vec{a} + c_{\vec{b}} \vec{b}}(q) = c_{\vec{a}} \mathcal{U}_{\vec{a}}(q) + c_{\vec{b}} \mathcal{U}_{\vec{b}}(q).
\end{align*}
Sums of vectors in this way will always be taken as formal sums. With this notation, we have the following result of Yuan and Zhao, generalizing that of Bachmann and K\"{u}hn; see \cite[Proposition 6.4, Theorem 6.5]{YuanZhao}. 

\begin{theorem} \label{T: Quasi-shuffle and q-MZV} \label{T: Quasishuffle structure}
    Let $[\cdot] : \lp \Q_N\langle A \rangle, * \rp \to \lp \mathcal{MD}_q(N), \cdot \rp$ by the $\Q\lp \zeta_N \rp$-linear map defined by
    \begin{align*}
        \left[ z_{\vec{a},\vec{c}} \right] := \left[ z_{a_1,c_1} \cdots z_{a_\ell, c_\ell} \right] := \mathcal{U}_{\vec{a}; \vec{c}}(q).
    \end{align*}
    Then $[\cdot]$ is an algebra homomorphism. That is, if we let $\vec{a}, \vec{c}$ and $\vec{b},\vec{d}$ be vectors of non-negative integers having lengths $r$ and $s$, respectively, then the product of two $q$-multiple divisor sums at level $N$ can be computed by the recursive rules $[1*z_{\vec{a}+\vec{1},\vec{c}+\vec{1}}] = [z_{\vec{a}+\vec{1},\vec{c}+\vec{1}}*1] = [z_{\vec{a}+\vec{1},\vec{c}+\vec{1}}] = \lp \prod_{j=1}^{r} a_j! \rp^{-1} \mathcal{U}_{\vec{a},\vec{c}}(q)$ and
    \begin{align*}
        \bigg( &\prod_{j=1}^{r} \lp a_j - 1 \rp! \cdot \prod_{j=1}^{s} \lp b_j - 1 \rp! \bigg)^{-1} 
        \mathcal{U}_{\vec{a};\vec{c}}(q) \ \mathcal{U}_{\vec{b};\vec{d}}(q) = [z_{\vec{a}+\vec{1},\vec{c}+\vec{1}} * z_{\vec{b}+\vec{1},\vec{d}+\vec{1}}] \\ &= \left[ z_{a_1+1,c_1+1}\lp z_{(a_2,\dots,a_r)+\vec{1},(c_2,\dots,c_r)+\vec{1}} * z_{\vec{b}+\vec{1},\vec{d}+\vec{1}} \rp \right] + \left[ z_{b_1+1,d_1+1}\lp z_{\vec{a}+\vec{1},\vec{c}+\vec{1}} * z_{(b_2,\dots,b_s)+\vec{1},(d_2,\dots,d_s)+\vec{1}} \rp \right] \\ &+ \left[ \lp z_{a_1+1,c_1+1} \diamond z_{b_1+1,d_1+1} \rp \lp z_{(a_2,\dots,a_r)+\vec{1},(c_2,\dots,c_r)+\vec{1}} * z_{(b_2,\dots,b_s)+\vec{1},(d_2,\dots,d_s)+\vec{1}} \rp \right].
    \end{align*}
    Furthermore, $\mathcal{U}_{\vec{a};\vec{c}}(q) \ \mathcal{U}_{\vec{b};\vec{d}}(q)$ has highest weight $|\vec{a}| + |\vec{b}| + r + s$ and highest depth $r+s$.
\end{theorem}

\begin{example}
    We give the simplest example where both $q$-multiple divisor sums are defined at depth one; here, we have by \cite[Proposition 5.6]{YuanZhao} that
    \begin{align*}
        \mathcal{U}_{(s-1);(c)}(q) \mathcal{U}_{(t-1);(d)}&(q) = \mathcal{U}_{(s-1,t-1);(c,d)}(q) + \mathcal{U}_{(t-1,s-1);(d,c)}(q) + \dfrac{(s-1)! (t-1)! \delta_{c,d}}{\lp s+t+-1 \rp!} \mathcal{U}_{(s+t-1);(c)}(q) \\ &+ \sum_{j=1}^s \dfrac{(s-1)! (t-1)! \lambda_{s,t;c-d}^{j;N}}{(j-1)!} \mathcal{U}_{(j-1);(c)}(q) + \sum_{j=1}^t \dfrac{(s-1)! (t-1)! \lambda_{t,s;d-c}^{j;N}}{(j-1)!} \mathcal{U}_{(j-1);(d)}(q). 
    \end{align*}
    This identity can also be written as
    \begin{align} \label{Eq: Product formula}
        \notag \mathcal{U}_{(s);(c)}(q) &\mathcal{U}_{(t);(d)}(q) = \mathcal{U}_{(s,t);(c,d)}(q) + \mathcal{U}_{(t,s);(d,c)}(q) + \dfrac{s! t! \delta_{c,d}}{\lp s+t+1 \rp!} \mathcal{U}_{(s+t+1);(c)}(q) \\ &+ \sum_{j=0}^s \binom{s}{j} (-1)^t \omega_{s+t-j;c-d}^N \mathcal{U}_{(j);(c)}(q) + \sum_{j=0}^t \binom{t}{j} (-1)^s \omega_{s+t-j;d-c}^N \mathcal{U}_{(j);(d)}(q). 
    \end{align}
    We also note the formula in the level one case can be written (using \cite[Lemma 4.1]{YuanZhao}, see also \cite{BachmannKuhn,CraigIttersumOno})
    \begin{align} \label{Eq: Product formula level one}
        \notag \mathcal{U}_{(s)}(q) &\mathcal{U}_{(t)}(q) = \mathcal{U}_{(s,t)}(q) + \mathcal{U}_{(t,s)}(q) + \dfrac{s! t!}{\lp s+t+1 \rp!} \mathcal{U}_{(s+t+1)}(q) \\ &+ \sum_{j=0}^s \binom{s}{j} (-1)^t \dfrac{B_{s+t+1-j}}{s+t+1-j} \mathcal{U}_{(j)}(q) + \sum_{j=0}^t \binom{t}{j} (-1)^s \dfrac{B_{s+t+1-j}}{s+t+1-j} \mathcal{U}_{(j)}(q)
    \end{align}
    where $B_j$ are the classical Bernoulli numbers.
\end{example}

By using the inductive structure of the quasi-shuffle product, Theorem \ref{T: Quasi-shuffle and q-MZV} implies that any product of any number of $q$-multiple divisor sums of level $N$ can be written as a $\Q\lp \zeta_N \rp$-linear combinations of elements of $\mathcal{MD}_q(N)$. We also briefly note that Yuan and Zhao also compute a precise formula for the action of the differential operator $D$ on $\mathcal{MD}_q(N)$ in \cite[Theorem 7.8]{YuanZhao}. We do not explicitly use this formula and it is rather involved, and so choose not restate the formula here.

\section{Proof of Main Results} \label{S: Main Proofs}

In this section, we prove Theorems \ref{T: Main Theorem, Cyclotomic} and \ref{T: Main Theorem, Integral} using the machinery of quasi-shuffle algebras along with the structure theorems for the algebras of quasimodular forms stated previously.

\subsection{Proof of Theorem \ref{T: Main Theorem, Cyclotomic}}

By either Corollary \ref{C: Construction 1}, we need only construct the Eisenstein series of weight 1 and level $N \geq 3$ using $q$-multiple zeta values of level $N$. We focus on Corollary \ref{C: Construction 1}, which requires that all weight one forms and $E_2(z)$ be constructed this way. Firstly, we recall the series expansions, for $\lambda = (a,b) \in \Z^2/(N\Z)^2$,
\begin{align*}
    E_{1,\lambda,N}(z) = - \dfrac{1}{\lp 2\pi i \rp N} \bigg[ \dfrac{1}{2} Z_{b,N}^1\lp - \dfrac 12 \rp &+ \sum_{\ell > 0} \sum_{\substack{m>0 \\ m \equiv a \bmod{N}}} \zeta_N^{b\ell} q^{m\ell/N} - \sum_{\ell > 0} \sum_{\substack{m>0 \\ m \equiv -a \bmod{N}}} \zeta_N^{-b\ell} q^{m\ell/N} \bigg].
\end{align*}
Recalling also that
\begin{align*}
    \sigma_{(0),(a)}(n) = \sum_{m \ell = n} \zeta_N^{am},
\end{align*}
we have using orthogonality of roots of unity that for any $a$ and for $d = \gcd(a,N)$, letting $a^{-1} a \equiv 1 \pmod{N/d}$, we have
\begin{align*}
    \dfrac{1}{N/d} \sum_{c=1}^{N/d} \zeta_{N/d}^{-(a/d)c} \sigma_{(0);(a)}(n) \lp \zeta_{N}^{bd a^{-1}} q \rp^n = \sum_{\substack{m\ell = n \\ m \equiv a \bmod{N}}} \zeta_N^{b\ell} q^{m\ell}.
\end{align*}
Therefore, we have
\begin{align*}
    \mathcal{S}_{a,b,N}(q) := \dfrac{d}{N} \sum_{c=1}^{N/d} \zeta_{N}^{-ac} \mathcal{U}_{(0);(a)}\lp \zeta_N^{bd a^{-1}} q \rp = \sum_{\ell > 0} \zeta_N^{b\ell} \sum_{\substack{m\ell = n \\ m \equiv a \bmod{N}}} q^{m\ell}.
\end{align*}
Thus, for $\lambda = (a,b)$ with $a$ coprime to $N$, it is clear that
\begin{align*}
    E_{1,\lambda,N}(z) = - \dfrac{1}{2\pi i N} \left[ \dfrac{1}{2} M_{b,N}^1\lp - \dfrac 12 \rp + \mathcal{S}_{a,b,N}\lp q^{1/N} \rp - \mathcal{S}_{-a,b,N}\lp q^{1/N} \rp \right].
\end{align*}
As we consider constant functions to be the depth zero $q$-multiple divisor sums, this demonstrates that each weight one Eisenstein series (with $a$ coprime to $N$) is formed from $q$-multiple divisor sums after the substitution $q \mapsto q^{1/N}$. Since the Eisenstein space of weight one is spanned by such forms \cite{Shimura}, and so we need only consider $E_2(z)$, which follows from Proposition \ref{P: Quasimodular level one surjection} since the level one $q$-multiple zeta values also belong to level $N$. \hfill $\qed$

\subsection{Proof of Theorem \ref{T: Main Theorem, Integral}}

Because of Corollary \ref{C: Integer weight 2 generation}, we know that the generation of the algebra $\cQM\lp \Gamma, \Z \rp$ occurs at the same weight, namely 2, as the algebra $\cQM\lp \Gamma, \Q\lp \zeta_N \rp \rp$. Thus since Theorem \ref{T: Main Theorem, Cyclotomic} is established, Theorem \ref{T: Main Theorem, Integral} follows, i.e. the weight $\leq 2$ elements of $\cQM\lp \Gamma, \Z \rp$ generate this entire algebra. Because we know how to generate Eisenstein series at weight one as $q$-multiple zeta values of level $N$, the proof follows. To be more specific, the quasimodular forms of level $N$ and weights $k \geq 2$ are generated as in Corollary \ref{C: Integer weight 2 generation}. \hfill $\qed$

\subsection{Proof of Theorem \ref{T: Main Theorem, Symmetry}}

First note that the product of two $S_a^\pm$-traces is the trace of products, and that by Theorem \ref{T: Quasi-shuffle and q-MZV} these products are reducible to linear combinations. Thus, closure of the space of symmetric forms as an algebra follows. 

We will now prove the quasimodularity of each term under the substitution $q \mapsto q^{1/N}$, beginning with the cases of depth one and two. For depth one, note that if $G_{k,\lambda,N}(z)$ denotes the level $N$ Eisenstein series with constant term removed and leading coefficient one, we have for $k \geq 0$, $b \in \Z$ that
\begin{align} \label{Eq: Simple quasimodular case}
    \sum_{a=0}^{N-1} G_{k+1,(a,b),N}(z) &= \sum_{n>0} \sum_{m\ell = n} \zeta_N^{b\ell} \ell^k q^{n/N} + (-1)^{k+1} \sum_{n>0} \sum_{m\ell = n} \zeta_N^{-b\ell} \ell^k q^{n/N} \\ \notag &= \mathcal{U}_{(k);(b)}\lp q^{1/N} \rp + (-1)^{k+1} \mathcal{U}_{(k);(-b)}\lp q^{1/N} \rp = \mathcal{U}^{\mathrm{sym}}_{(k);(b)}\lp q^{1/N} \rp
\end{align}
is a quasimodular form of level $N$. For convenience, we will now suppress $q^{1/N}$ from the remainder of the proof. We now consider the depth two case. We first observe that
\begin{align*}
    &\lp \mathcal{U}_{(k_1);(c)} + (-1)^{k_1+1} \mathcal{U}_{(k_1);(-c)} \rp \lp \mathcal{U}_{(k_2);(d)} + (-1)^{k_2+1} \mathcal{U}_{(k_2);(-d)} \rp \\ &= \mathcal{U}_{(k_1);(c)} \mathcal{U}_{(k_2);(d)} + \lp -1 \rp^{k_1+1} \mathcal{U}_{(k_1);(-c)} \mathcal{U}_{(k_2);(d)} + \lp -1 \rp^{k_2+1} \mathcal{U}_{(k_1);(c)} \mathcal{U}_{(k_2);(-d)} + \lp -1 \rp^{k_1 + k_2} \mathcal{U}_{(k_1);(-c)} \mathcal{U}_{(k_2);(-d)}.
\end{align*}
Now, using \eqref{Eq: Product formula}, we see that for any $c,d \pmod{N}$, the product $\mathcal{U}_{(k_1);(c)} \mathcal{U}_{(k_2);(d)}$ takes the form
\begin{align*}
     \mathcal{U}_{(k_1);(c)} &\mathcal{U}_{(k_2);(d)} = \mathcal{U}_{(k_1,k_2);(c,d)} + \mathcal{U}_{(k_2,k_1);(d,c)} + \dfrac{k_1! k_2! \delta_{c,d}}{\lp k_1+k_2+1 \rp!} \mathcal{U}_{(k_1+k_2+1);(c)} \\ &+ \sum_{j=0}^{k_1} \binom{k_1}{j} (-1)^{k_2} \omega_{k_1+k_2-j;c-d}^N \mathcal{U}_{(j);(c)} + \sum_{j=0}^{k_2} \binom{k_2}{j} (-1)^{k_1} \omega_{k_1+k_2-j;d-c}^N \mathcal{U}_{(j);(d)}. 
\end{align*}
Now, such quasi-shuffle product formulas apply for each of the four terms above. We wish to eliminate terms which are already known to be quasimodular. We consider first the terms of depth one, i.e. each term $\mathcal{U}_{(j);(c)}$. In the sums coming from the sum $0 \leq j \leq k_1$, we see from each $j$ and each of the four terms, the overall contribution is a multiple of
\begin{align*}
    (-1)^{k_2} &\lp \omega^N_{k_1+k_2-j;c-d} + (-1)^{k_2+1} \omega^N_{k_1+k_2-j;c+d} \rp \mathcal{U}_{(j);(c)} \\ &+ (-1)^{k_1+k_2} \lp \omega^N_{k_1+k_2-j;-c+d} + (-1)^{k_2+1} \omega^N_{k_1+k_2-j;-c-d} \rp \mathcal{U}_{(j);(-c)}.
\end{align*}
Now, it is easy to see that $\omega_{n;c}^N = (-1)^{n+1} \omega_{n;-c}^N$ for $n \geq 1$, the contribution is then a multiple of
\begin{align*}
    \mathcal{U}_{(j);(c)} + (-1)^{j+1} \mathcal{U}_{(j);(-c)},
\end{align*}
which is a quasimodular form. Thus, the terms which arise in the product arising from either sum over $j$ (by symmetry of argument) are quasimodular forms. If now $c \not \equiv \pm d \pmod{N}$, no contribution of terms $\mathcal{U}_{(k_1+k_2+1);(c)}$ occurs. If one of these conditions holds, then a term $\mathcal{U}_{(k_1+k_2+1);(c)} + (-1)^{k_1+k_2} \mathcal{U}_{(k_1+k_2+1);(-c)}$ naturally arises in the full product, which is also quasimodular. Thus, the remaining pieces, namely
\begin{align*}
    &\mathcal{U}_{(k_1,k_2);(c,d)} + \mathcal{U}_{(k_2,k_1);(d,c)} + (-1)^{k_1+1} \mathcal{U}_{(k_1,k_2);(-c,d)} + (-1)^{k_1+1} \mathcal{U}_{(k_2,k_1);(d,-c)} \\ &+ (-1)^{k_2+1} \mathcal{U}_{(k_1,k_2);(c,-d)} + (-1)^{k_2+1} \mathcal{U}_{(k_2,k_1);(-d,c)} + (-1)^{k_1+k_2} \mathcal{U}_{(k_1,k_2);(-c,-d)} + (-1)^{k_1+k_2} \mathcal{U}_{(k_2,k_1);(-d,-c)},
\end{align*}
necessarily form a quasimodular form. Observe then that this is exactly the traceform $\mathcal{U}_{(k_1,k_2);(c,d)}^{\mathrm{sym}}$; we have thus proven the desired result for depths one and two.

We now complete the proof by induction on the depth of the $q$-multiple zeta value. We begin with the particular $q$-multiple zeta value $\mathcal{U}_{\vec{k};\vec{c}}$ with $\vec{k} = \lp k_1, \dots, k_a \rp \in \Z_{\geq 0}^a$ and $\vec{c} = \lp c_1, \dots, c_a \rp \in \lp \Z/N\Z \rp^a$. Recall from Theorem \ref{T: Quasi-shuffle and q-MZV} the notation $[z_{\vec{k},\vec{c}}] = [z_{k_1,c_1} z_{k_2,c_2} \cdots z_{k_a,c_a}] = \lp \prod_{j=1}^a (k_j-1)! \rp^{-1} \mathcal{U}_{\vec{k};\vec{c}}$ and the quasi-shuffle product identity
\begin{align*}
    [z_{\vec{a},\vec{c}} * z_{\vec{b},\vec{d}}] &= \left[ z_{a_1,c_1}\lp z_{(a_2,\dots,a_r),(c_2,\dots,c_r)} * z_{\vec{b},\vec{d}} \rp \right] + \left[ z_{b_1,d_1}\lp z_{\vec{a},\vec{c}} * z_{(b_2,\dots,b_s),(d_2,\dots,d_s)} \rp \right] \\ &+ \left[ \lp z_{a_1,c_1} \diamond z_{b_1,d_1} \rp \lp z_{(a_2,\dots,a_r),(c_2,\dots,c_r)} * z_{(b_2,\dots,b_s),(d_2,\dots,d_s)} \rp \right].
\end{align*}
It is straightforward to extend the $S_a^\pm$ group action to $[z_{\vec{k},\vec{c}}]$ as
\begin{align*}
    \lp \sigma, s \rp \cdot [z_{\vec{k},\vec{c}}] = (-1)^{\varepsilon_s(\vec{k},\vec{c})} [ z_{\sigma \vec{k}, s \sigma \vec{c}}].
\end{align*}
We have assumed by way of induction that the symmetric forms $\mathcal{U}_{\vec{k};\vec{c}}^{\mathrm{sym}}$ for $\vec{k},\vec{c}$ of depth $< a$. Now, using the quasi-shuffle product above, we see that
\begin{align} \label{Eq: Induction Equation}
    [z_{k_1,c_1} * z_{(k_2,\dots,k_a),(c_2,\dots,c_a)}] &= \left[ z_{\vec{k},\vec{c}} \right] + \left[ z_{k_2,c_2} \lp z_{k_1,c_1} * z_{(k_3,\dots,k_a),(c_3,\dots,c_a)} \rp \right] \\ \notag &+ \left[ \lp z_{k_1,c_1} \diamond z_{k_2,c_2} \rp z_{(k_3,\dots,k_a),(c_3,\dots,c_a)} \right].
\end{align}
Now, by induction, we know that \eqref{Eq: Induction Equation} can be expressed as a linear combination of forms $\mathcal{U}_{\vec{a};\vec{b}}$ of depths $\leq a$. Once we apply the trace of the action of $S_a^\pm$, all terms of depth $<a$ will be quasimodular by induction, and so we only need to consider terms of depth exactly $a$. The first term $[z_{\vec{k},\vec{c}}]$ is of course of depth exactly $a$. Because the product $z_{a,b} \diamond z_{c,d}$ remains in depth one, the third term can now be ignored. In the second term, brackets of depth $a$ arise exactly from terms of $[z_{k_1,c_1} * z_{(k_3,\dots,k_a),(c_3,\dots,c_a)}]$ having depth $a-1$. By descending inductively, we see that the contributing terms all have the form
\begin{align*}
    [z_{(k_j,\dots,k_1,k_{j+1},\dots,k_a),(c_j,\dots,c_1,c_{j+1},\dots,c_a)}].
\end{align*}
But after application of the trace, all contributing terms become multiples of $\mathcal{U}_{\vec{k};\vec{c}}^{\mathrm{sym}}$. Thus, since the left-hand side of \eqref{Eq: Induction Equation} is quasimodular after taking traces, so $\mathcal{U}_{\vec{k};\vec{c}}$ must be quasimodular after taking traces, i.e. $\mathcal{U}_{\vec{k};\vec{c}}^{\mathrm{sym}}$ is quasimodular This completes the proof of Theorem \ref{T: Main Theorem, Symmetry} (1).

We now turn to Theorem \ref{T: Main Theorem, Symmetry} (2), for which we follow the concept of the proof of \cite[Theorem 4.4 (3)]{CraigIttersumOno}. It is easy to see, just as in the proof of Theorem \ref{T: Main Theorem, Cyclotomic}, that we may generate each Eisenstein series of level $N$ and weight one by using twists of $\mathrm{U}^{\mathrm{sym}}_{(0);(c)}(q^{1/N})$. Thus, the algebra of twisted symmetric $q$-multiple zeta values is surjective onto the space of quasimodular forms of level $N$ by Corollary \ref{C: Construction 1}, and by the closure of the symmetric space under products already shown in (1), the algebra generation may be reduced to linear generation. The final aspect of this theorem follows from Lemma \ref{L: Sturm idea}. \hfill $\qed$

\section{Crank moments and the Atkin--Garvan Rank-Crank PDE} \label{S: Partition cranks}

Undoubtedly one of the central elements of the modern theory of partitions are the theories built around Dyson's rank and the Andrews--Garvan crank statistics. In this section, we demonstrate how the multiple divisor sums may be used to compute these important statistics without any complicated convoluton sums. For a partition $\lambda$, the rank statistic $\mathrm{rk}\lp\lambda\rp$ is defined as the largest part minus the number of parts, and the crank statistic $\mathrm{ck}(\lambda)$ is defined as the largest part of $\lambda$ if $\lambda$ has no ones, and as the number of parts larger than the number of ones minus the number of ones otherwise. As is customary, we let $N(m,n)$ and $M(m,n)$ denote the number of partitions of $n$ whose rank and crank are $m$, respectively\footnote{We make the usual adjustments $M(-1,1) = M(0,1) = M(1,1)$ so that $M(m,n)$ fits the given infinite product.}; it is well known that the generating functions of these sequences are
\begin{align*}
	R(z,q) &:= \sum_{n \geq 0} \sum_{m \in \Z} N(m,n) z^m q^n = 1 + \sum_{n \geq 1} \dfrac{q^{n^2}}{\lp zq;q \rp_n \lp z^{-1}q; q \rp_n}, \\
	C(z,q) &:= \sum_{n \geq 0} \sum_{m \in \Z} M(m,n) z^m q^n = \dfrac{\lp q;q \rp_\infty}{\lp zq;q \rp_\infty \lp z^{-1}q;q \rp_\infty}.
\end{align*}
For fixed $m$, one also has the generating functions
\begin{align*}
	\sum_{n \geq 0} N(m,n) q^n &= \dfrac{1}{\lp q;q \rp_\infty} \sum_{n \geq 1} \lp -1 \rp^{n-1} q^{\frac{n(3n-1)}{2} + |m|n} \lp 1 - q^n \rp, \\
	\sum_{n \geq 0} M(m,n) q^n &= \dfrac{1}{\lp q;q \rp_\infty} \sum_{n \geq 1} (-1)^{n-1} q^{\frac{n(n-1)}{2} + |m|n} \lp 1 - q^n \rp.
\end{align*}
The statistics give the first combinatorial explanations of the famous Ramanujan congruences
\begin{align*}
	p(5n+4) \equiv 0 \pmod{5}, \ \ \ p(7n+5) \equiv 0 \pmod{7}, \ \ \ p(11n+6) \equiv 0 \pmod{11}.
\end{align*}
More specifically, consider the functions $N(k,t,n)$ and $M(k,t,n)$ that count the number of partitions of $n$ with rank (resp. crank) congruent to $k$ modulo $t$. Dyson conjectured \cite{Dyson} and Atkin and Swinnerton-Dyer proved \cite{ASD} that for $t = 5,7$, $24 \delta_t \equiv 1 \pmod{24}$, we have
\begin{align*}
	N\lp k,t,tn+\delta_t \rp = \dfrac{p(n)}{t}, \ \ \ 0 \leq k \leq t-1,
\end{align*}
so that partitions in the designated congruence classes are split into $t$ equally sized classes. Likewise, Dyson conjectured that a ``crank statistic" should exist which would likewise explain all three congruences. This crank is exactly the one previously defined, and was discovered by Garvan \cite{Garvan} and Andrews--Garvan \cite{AndrewsGarvan}.
The rank and crank statistics and their generalizations \cite{GKS} play an absolutely foundational role not only in the arithmetic study of partitions, but also in the study of $q$-series. Of special importance is the paper of Bringmann and Ono \cite{BringmannOno} demonstrating that $R(z,q)$ and $C(z,q)$ fit within the theory of mock modular forms and harmonic Maass forms as initiated by Zwegers \cite{Zwegers} and Brunier and Funke \cite{BruinierFunke}. See \cite{HMFBook,ZagierMock} for further references.

Among the many important developments which emerge from the theories of ranks and cranks, we focus now on the foundational work of Atkin and Garvan \cite{AtkinGarvan} which relates the two statistics. The central theme of their work is the so-called {\it rank-crank PDE}, a partial differential equation which relates the generating functions $R(z,q)$ and $C(z,q)$ and which produces identities for their moments. This partial differential equation is given \cite[Theorem 1.1]{AtkinGarvan} as
\begin{align*}
	z \lp q;q \rp_\infty^2 \left[ C(z,q) \right]^3 = \lp 3(1-z)^2 D_q + \dfrac{1}{2} \lp 1-z \rp^2 D_z^2 - \dfrac{1}{2}\lp z^2-1 \rp D_z + z \rp R(z,q).
\end{align*}
Atkin and Garvan and subsequent authors have used this partial differential equation to construct large families of identities and congruence relations relating the rank and crank via their moments. In particular, if we define the $j$th moments by
\begin{align*}
	N_j(n) := \sum_{m \in \Z} m^j N(m,n), \ \ \ M_j(n) := \sum_{m \in \Z} m^2 M(m,n),
\end{align*}
they discover numerous identities relating $N_j(n)$ and $M_j(n)$. They demonstrate, in fact, that there are polynomials $P_k(n)$ and $Q_{k,j}(n)$ of degrees $k-1$ and $k-j$, respectively, such that for each $k \geq j \geq 1$, we have
\begin{align*}
    N_{2k}(n) = P_k(n) N_2(n) + \sum_{j=1}^k Q_{k,j}(n) M_{2j}(n).
\end{align*}
If we let $R_j(q)$ and $C_j(q)$ denote the generating functions for the $j$th rank and crank moments, respectively, such identities lift to $q$-series as
\begin{align*}
    R_j(q) = P_k(D) R_2(q) + \sum_{j=1}^k Q_{k,j}(D) C_{2j}(D).
\end{align*}
An example of such an identity is given in \cite[Equation 1.32]{AtkinGarvan} as
\begin{align*}
    N_4(n) = - \lp 2n + \frac 23 \rp M_2(n) + \frac 83 M_4(n) + \lp 1 - 12n \rp N_2(n).
\end{align*}

Central to this paper are the divisor sum functions
\begin{align*}
    \mathcal{U}_{(j)}(q) = \sum_{m,n \geq 1} m^j q^{mn} = \sum_{n \geq 1} \sigma_j(n) q^n.
\end{align*}
The centrality of such $q$-series comes from the differential structure of partition generating functions and the structure of the space of quasimodular forms, of which their paper is one of the earliest systematic works. Indeed, they give differential formulas such as \cite[Equation 3.23]{AtkinGarvan}
\begin{align} \label{Eq: Derivative Atkin-Garvan}
    D\lp \mathcal{U}_{(5)}(q) \rp = \dfrac{1}{2} \mathcal{U}_{(5)}(q) + \dfrac{1}{42} \mathcal{U}_{(1)}(q) - 12 \mathcal{U}_{(1)}(q) \mathcal{U}_{(5)}(q) + \dfrac{10}{21} \mathcal{U}_{(3)}(q) + \dfrac{400}{7} \mathcal{U}_{(3)}(q)^2
\end{align}
which we may view as a precursor to the differential structure of spaces like $\mathcal{MD}_q(N)$. Their work even studies vector spaces of products of the $\mathcal{U}_{(j)}(q)$ series \cite[Theorem 3.5]{AtkinGarvan} and in this way contains precursors to Bachmann and K\"{u}hn's algebraic theory of $q$-multiple divisor sums.

The main application of the rank-crank PDE in \cite{AtkinGarvan} is the computation of $C_j(q)$ in terms of quasimodular forms. This is accomplished by viewing moments through the lens of the recurrence\footnote{By symmetry, crank and rank moments vanish for odd $j$, and therefore these are ignored and $\frac j2 \in \Z$.} (see \cite[Equations 4.5-4.6]{AtkinGarvan})
\begin{align*}
	C_a(q) = D_z^a \lp C(z,q) \rp \bigg|_{z=1} = \sum_{j=1}^{\frac{a}{2}-1} \binom{a-1}{2j-1} \mathcal{U}_{(2j)}(q) C_{a-2j}(q) + 2 \mathcal{U}_{(a)}(q) P(q)
\end{align*}
where we recall $P(q) := \sum_{n \geq 0} p(n) q^n = \lp q;q \rp_\infty^{-1}$. Applying this recurrence, they give exact formulas \cite[Equation 4.7]{AtkinGarvan} such as 
\begin{align} \label{Eq: Crank moments}
    C_4(q) &= 2 P(q) \lp \mathcal{U}_{(3)}(q) + 6 \mathcal{U}_{(1)}(q)^2 \rp, \\
    \notag C_6(q) &= 2 P(q) \lp \mathcal{U}_{(5)}(q) + 30 \mathcal{U}_{(3)}(q) \mathcal{U}_{(1)}(q) + 60 \mathcal{U}_{(1)}(q)^3 \rp.
\end{align}
In Atkin and Garvan's work, they express such identities in terms of the space $\mathcal{W}_n$ generated linearly by forms $\mathcal{U}_{(1)}^a \mathcal{U}_{(3)}^b\mathcal{U}_{(5)}^c$ for $0 < a+2b+3c \leq n$. In the language of $q$-multiple zeta values, we may translate this space into $q$-multiple zeta value spaces, in which the convolution sums may be removed by increasing the depth in accordance with the corresponding quasi-shuffle product. In particular, the following result follows immediately from \cite[Lemma 4.1, Theorem 4.2]{AtkinGarvan} along with Theorem \ref{T: Main Theorem, Integral}.

\begin{corollary}
    Let $m,a \geq 1$ be integers. Then there exists $\Phi_1, \Phi_2 \in \mathcal{MD}_q(1)$ of weights $\leq 2m$ and depths $\leq m$ such that
    \begin{align*}
		D^m P(q) = \Phi_1 P(q)
    \end{align*}
	and
    \begin{align*}
        D^m C_{2a}(q) = \Phi_2 P(q).
    \end{align*}
\end{corollary}

Atkin and Garvan give as one example the formula
\begin{align*}
    D \lp C_2(q) \rp &= - \dfrac{1}{3} P(q) \lp 6 \mathcal{U}_{(1)}(q)^2 - 5 \mathcal{U}_{(3)}(q) - \mathcal{U}_{(1)}(q) \rp.
\end{align*}

These identities may be linearized by use of the quasi-shuffle structure given in \eqref{Eq: Product formula level one}. To demonstrate the kinds of identities which arise from the quasi-shuffle structure of $q$-multiple divisor sums, we give the following additive identities which remove convolutions from the two crank moment identities stated in this section:
\begin{align*}
    C_4(q) &= 2 P(q) \left[ 6 \mathcal{U}^{\mathrm{sym}}_{(1,1)}(q) + 2 \mathcal{U}_{(3)}(q) - \mathcal{U}_{(1)}(q) \right], \\
    C_6(q) &= 2 P(q) \left[ 60 \mathcal{U}_{(1,1,1)}^{\mathrm{sym}}(q) + 60 \mathcal{U}_{(3,1)}^{\mathrm{sym}}(q) - 30 \mathcal{U}_{(1,1)}^{\mathrm{sym}}(q) + 3  \mathcal{U}_{(5)}(q) - 5 \mathcal{U}_{(3)}(q) + 3 \mathcal{U}_{(1)}(q) \right], \\
    D\lp C_2(q) \rp &= - \dfrac 13 P(q) \left[ 6 \mathcal{U}_{(1,1)}^{\mathrm{sym}}(q) - 4 \mathcal{U}_{(3)}(q) - 2 \mathcal{U}_{(1)}(q) \right].
\end{align*}
Here, we have used here the symmetric traceforms of \cite{CraigIttersumOno} (which use only the $S_a$-action) rather than those of Theorem \ref{T: Main Theorem, Symmetry}, which are the more natural objects in level one and only differ from those in Theorem \ref{T: Main Theorem, Symmetry} by constants depending on the depth. Thus, only subscripts containing odd numbers only give symmetric trace quasimodular forms in level one. We also note that these formulas can be interpreted in terms of error to modularity; from \eqref{Eq: Product formula level one}, it follows that the depth two objects have combinations of odd weight non-modular Eisenstein series as their error to modularity. It would therefore be of interest to understand the error to modularity of odd weight Eisenstein series.

To finish this discussion, we provide fuller details in light of further work and theories developed since this time, in particular drawing from the theory of $q$-brackets and the work of Rhoades \cite{Rhoades} on crank moments. The $q$-bracket is defined for any function $f$ on partitions by a weighted average of $f$ over all partitions, i.e.
\begin{align*}
	\langle f \rangle_q := \dfrac{\sum_\lambda f(\lambda) q^{|\lambda|}}{\sum_\lambda q^{|\lambda|}} = \lp q;q \rp_\infty \sum_\lambda f(\lambda) q^{|\lambda|}.
\end{align*}
The quasimodularity of $q$-brackets has been studied by many authors, such as Bloch--Okounkov \cite{BlochOkounkov}, Zagier \cite{ZagierBracket}, Bachmann and van Ittersum \cite{
BachmannIttersum}, and van Ittersum \cite{Ittersum21, Ittersum23}. It turns out there are large algebras of functions on partitions, called symmetric and shifted symmetric functions, for which the $q$-bracket automatically yields quasimodular forms. These $q$-brackets are not only interesting devices for constructing quasimodular forms, but have important applications in representation theory (which indeed is the original context in which they arose \cite{BlochOkounkov}) and form an integral piece of current efforts to build a multiplicative theory of partitions \cite{Schneider17}. Thus, we emphasize in particular that these results on $C_a(q)$ not only introduce $q$-multiple zeta values to the study of cranks, but also naturally bring in the $q$-bracket. Using this language, we can conceptually explain the appearance of $P(q)$ in formulas relating to crank moments. We now state this result formally, using the work of Rhoades \cite[Theorem 2.1]{Rhoades} who made the results of Atkin and Garvan more explicit.\footnote{We also note the recent work of Amdeberhan, Griffin, Ono and Singh \cite{AGOS} who reproved and generalized Rhoades' result to the context of Jacobi forms with torsional divisor using their concept of partition Eisenstein series. We discuss their ideas in some detail in a later section.}

\begin{theorem}
	For $a \geq 1$, if we let $c_{2a}(\lambda) = \mathrm{crk}(\lambda)^{2a}$, then we have
	\begin{align*}
		  \langle c_{2a} \rangle_q = \sum_{1 \leq k \leq a} \sum_{\substack{i_1 + \dots + i_k = a \\ i_1, \dots, i_k > 0}} \dfrac{2^k \lp 2a \rp!}{k! \lp 2i_1 \rp! \cdots \lp 2i_k \rp!} \mathcal{U}_{(2i_1-1)}(q) \cdots \mathcal{U}_{(2i_k-1)}(q).
    \end{align*}
    Thus, $\langle c_{2a} \rangle_q$ is a mixed weight quasimodular form of highest weight $2a$, and $\langle c_{2a} \rangle_q$ is a $\Q$-linear combination of elements of $\mathcal{MD}_q(1)$ with highest weight $2a$ and highest depth $a$.
\end{theorem}

The proof follows immediately from \cite[Theorem 2.1]{Rhoades} along with the quasi-shuffle structure given in Theorem \ref{T: Quasi-shuffle and q-MZV}. These moments of ranks and cranks have received extensive study from many angles \cite{Andrews07}, and this new connection between crank moment formulas and $q$-multiple divisor sums deserves detailed attention. In particular, a natural questions arises from the results of Atkin and Garvan.

\begin{question} \label{Q: Ranks}
	Do the generating functions for rank moments belong to some algebra $\mathcal{MD}_q(N)$?
\end{question}

After Atkin and Garvan, much work has also been done for moments of cranks in arithmetic progressions, which require the study of twisted crank moments
\begin{align*}
	M_k\lp \zeta; n \rp := \sum_{m \in \Z} \zeta^m m^k M(m,n),
\end{align*}
These problems have been considered for example by \cite{AndrewsLewis,CKL}. Consider the generating function
\begin{align*}
	C_a\lp \zeta; q \rp := \sum_{n \geq 0} M_k\lp \zeta; n \rp q^n.
\end{align*}
As early as Rhoades \cite[Theorem 3.1]{Rhoades} gives a formula for $C_{2a}(-1;q)$ which, in light of Theorem \ref{T: Main Theorem, Cyclotomic} informs us that $C_{2a}(-1;q)$ is essentially a $q$-multiple zeta value of level 2. It would be interesting to understand Atkin--Garvan theory for congruence classes in the context of $q$-multiple zeta values and quasimodular forms of higher levels.

As a closing note on the collected literature on Atkin--Garvan crank moments, we observe that by \cite[Theorem 6.2]{AtkinGarvan} and Theorem \ref{T: YuanZhao Closure}, a positive answer to Question \ref{Q: Ranks} would imply that the half-integral weight modular form $\eta^{23}$ belongs to the algebra of $q$-multiple zeta values. If then Question \ref{Q: Ranks} is answered in the affirmative, it would forge the beginnings of a theory which could unify half-integral and integral weight modular forms under a single algebraic framework. If Question \ref{Q: Ranks} is answered in the negative, this would provide a new conceptual distinction between the rank and crank, as well as between integral and half-integral weight modular forms. We thus suggest Question \ref{Q: Ranks} as a crucial one.

\section{Prime-detecting $q$-multiple zeta values} \label{S: Primes}

Following the ideas of Schneider \cite{SchneiderThesis} mentioned in the introduction, there has been a growing literature exploring intersections with partition theory and multiplicative number theory; see for instance \cite{Schneider17,Schneider16,Densities1,Densities2}. Building further in this direction, the author, van Ittersum and Ono \cite{CraigIttersumOno} have demonstrated that there within the algebra of quasimodular forms of level one, there is an infinite-dimensional vector subspace of forms whose Fourier coefficients vanish exactly at the the constant term, linear term, and at $q^p$ for $p$ prime. In fact, the authors characterize exactly which of these forms can be constructed using linear combinations of derivatives of Eisenstein series, which conjecturally characterizes the whole space. Because of the connection established by Bachmann and K\"{u}hn between quasi-shuffle algebras, partitions, and $q$-multiple divisor sums, these theorems demonstrate an effective method to use partitions to detect primality of an integer. Thus, the goal of Schneider's thesis of connecting foundational concepts in partition theory and multiplicative number theory can be achieved. Gomez \cite{Gomez} has also shown how twists of these level one Eisenstein series produce prime-detecting partition-theoretic functions and even to detect cubes of primes.

We now will extend this theory using $q$-multiple zeta values of level $N$. In particular, we will show how to use forms of level $N$ to detect primes in arithmetic progressions, which is the natural next stage in the exploration of the interface between partitions, $q$-multiple zeta values, and multiplicative number theory. We first define what we mean by a prime-detecting quasimodular form. We call a $q$-series
\begin{align*}
    f(q) = \sum_{n \geq 0} a_f(n) q^n
\end{align*}
a {\it prime-detecting series of type} $(S,N)$ if the coefficients $a_f(n)$ are integers and, for $n \geq 2$, vanish if and only if $n \in S \pmod{N}$ and $n$ is prime.

We begin our search for such $q$-series with certain special $q$-multiple zeta values; for $k \geq 0$, $N \geq 1$, and $S \subset \{ 0, 1, \dots, N-1 \} \bmod{N}$, we consider the forms given by
\begin{align*}
    \widehat{\mathcal{U}}_{(k),S,N}(q) := \sum_{j \not \in S} \mathcal{U}_{(k),(j)}(q) = \sum_{n > 0} \sum_{\substack{d|n \\ d \not \in S \bmod{N}}} d^k q^n,
\end{align*}
which arise as combinations of special cases of the \eqref{Eq: Integral forms}. We extend the ideas of Gomez \cite{Gomez} (whose obtains results in terms of quasimodular forms with level $N^2$) with an emphasis how the quasimodularity at level $N$ of the original prime-detecting forms in \cite{CraigIttersumOno} can be regained by detecting primes in complementary pairs of arithmetic progressions. Our results here are thus natural complements to the results of Gomez on primes in arithmetic progressions.

Using the same style of argument as \cite[Lemma 2]{Lelievre},  \cite[Lemma 2.1]{CraigIttersumOno}, or \cite{Gomez} with a little additional elementary number theory, we obtain infinite families of $q$-multiple zeta values of level $N$ which are prime-detecting of type $(c,N)$. To state this result, and the natural quasimodular connections which follow from our approach, we recall the forms $f_{k,\ell}$ from \cite{CraigIttersumOno} defined by
\begin{align*}
    f_{k,\ell} := \lp D^\ell + 1 \rp \mathcal{U}_{(k)} - \lp D^k + 1 \rp \mathcal{U}_{(\ell)}.
\end{align*}

\begin{proposition} \label{L: Partial prime detection}
    Let non-negative integers $k,\ell$ with $\ell > k$, and a modulus $N \geq 3$ be given. For this $N$, consider the modified differential operator $D_N := N D$ (which takes $q^{n/N} \mapsto n q^{n/N}$). For $c \bmod{N}$, define the $q$-multiple zeta values
    \begin{align*}
        f_{k,\ell,c,N}(q^{1/N}) := \begin{cases}
            \lp D_N^\ell + 1 \rp \widehat{\mathcal{U}}_{(k),\{ c \},N}(q^{1/N}) - \lp D_N^k + 1 \rp \widehat{\mathcal{U}}_{(\ell),\{ c \},N}(q^{1/N}) & \text{if } c \not \equiv 1 \bmod{N}, \\
            \lp D_N^\ell + 1 \rp \widehat{\mathcal{U}}_{(k),1,N}(q^{1/N}) - \lp D_N^k + 1 \rp \widehat{\mathcal{U}}_{(\ell),1,N}(q^{1/N}) + f_{k,\ell}(q^{1/N}) & \text{if } c \equiv 1 \bmod{N}.
        \end{cases}
    \end{align*}
    and let $f_{k,\ell,c,N}$ be defined by linearity in the subscript. Let $m \geq 0$ be given. Then the following are true:
    \begin{enumerate}
        \item $D_N^m f_{k,\ell,S,N}(q^{1/N})$ is a prime-detecting $q$-series of type $(c,N)$.
        \item If $S$ is symmetric (i.e. $S = -S \bmod{N}$) and $k, \ell$ are odd, then $D_N^m f_{k,\ell,S,N}(q^{1/N})$ is a prime-detecting quasimodular form of type $(S,N)$.
        \item If $k \equiv \ell \bmod{2}$, then each
    \begin{align*}
        D_N^m \lp f^{\mathrm{sym}}_{k,\ell,S,N}(q^{1/N}) \rp := D_N^m\lp f_{k,\ell,S,N}(q^{1/N}) + (-1)^{k+1} f_{k,\ell,-S,N}(q^{1/N}) \rp
    \end{align*}
    is a prime-detecting quasimodular form of type $\lp S, N \rp$.
    \end{enumerate} 
\end{proposition}

\begin{proof}
    The proofs go exactly as in \cite[Lemma 2.1]{CraigIttersumOno}, except that we need to justify the quasimodularity of $f^{\mathrm{sym}}_{k,\ell,c,N}$. This follows from Theorem \ref{T: YuanZhao Closure} and the fact, from Theorem \ref{T: Main Theorem, Symmetry}, that
    \begin{align*}
        \mathcal{U}_{(k);(b)}(q^{1/N}) + (-1)^{k+1} \mathcal{U}_{(k);(-b)}(q^{1/N})
    \end{align*}
    is a quasimodular form of level $N$.
\end{proof}

Thus, there are large vector spaces of $q$-multiple zeta values of level $N$ that detect primes in arithmetic progressions modulo $N$, and also of quasimodular $q$-multiple zeta values of level $N$ detecting pairs of progressions.
We close the section with the following open question, which is the next natural generalization of the question resolved by our work and that of Gomez:

\begin{question}
    Let $K/\Q$ be a finite Galois extension. Is it possible to construct, using $q$-multiple zeta values or generalizations thereof, a $q$-series whose coefficients detect primes fixed by the Frobenius element of the Galois group $\mathrm{Gal}\lp K/\Q \rp$? If so, which elements generate these spaces? If so, can we impose the additional condition that these $q$-series be modular objects of some kind?
\end{question}

There is already strong evidence of many other sets related to primes being detectable by partitions. For example, Gomez \cite{Gomez} provides a mechanism for detecting cubes of primes by partitions. We thus also leave the following open question generalizing the ideas of Gomez.

\begin{question}
    For any integer $k \geq 1$, can one create partition-theoretic series whose Fourier coefficients detect $k$th powers of primes? If so, are these relatable to $q$-multiple zeta values and/or quasimodular forms?
\end{question}

\section{Ramanujan's $\tau$-function} \label{S: Ramanujan Tau}

We now consider the potential for applications to some classical problems in modular forms; we choose as a prime example Lehmer's conjecture. Recall the famous Ramanujan $\Delta$-function, and its coefficients the Ramanujan $\tau$-function, defined by
\begin{align*}
    \Delta(z) := \sum_{n \geq 1} \tau(n) q^n := q \prod_{n \geq 1} \lp 1 - q^n \rp^{24}.
\end{align*}
This is the unique normalized cusp form of weight 12. It is easy to see, using the uniqueness of $\Delta$ and the simple identities $E_4^2(z) = E_8(z)$ and $E_k(z) = 1 - \frac{2k}{B_k} \mathcal{U}_{(k-1)}(q)$ that
\begin{align*}
    \Delta(z) &= \dfrac{E_4^3(z) - E_6^2(z)}{1728} = \dfrac{E_4(z) E_8(z) - E_6^2(z)}{1728} \\ &= \dfrac{200}{3} \mathcal{U}_{(3)}(q) \mathcal{U}_{(7)}(q) - 147 \mathcal{U}_{(5)}(q)^2 + \dfrac{5}{18} \mathcal{U}_{(7)}(q) + \dfrac{7}{12} \mathcal{U}_{(5)}(q) + \dfrac{5}{36} \mathcal{U}_{(3)}(q).
\end{align*}
Therefore, using \eqref{Eq: Product formula level one} we may obtain
\begin{align*}
    \Delta(z) &= \dfrac{200}{3} \mathcal{U}_{(7,3)}^{\mathrm{sym}}(q) - 147 \mathcal{U}_{(5,5)}^{\mathrm{sym}}(q) - \dfrac{1}{396} \mathcal{U}_{(11)}(q) + \dfrac{5}{6} \mathcal{U}_{(7)}(q) - \dfrac{137}{36} \mathcal{U}_{(5)}(q) - \dfrac{19}{9} \mathcal{U}_{(3)}(q) + \dfrac{1205}{198} \mathcal{U}_{(1)}(q).
\end{align*}
We note also the work of \cite{Tasaka}, who computed expressions for eigenforms at level one using double Eisenstein series, wherein formulas analogous to this can be found for many forms. We observe that such exact formulas with generalizations of divisor sums suggest connections with the famous Galois congruences exist modulo powers of 2, 3, 5, 7, 23, and 691. This new formula suggests a connection between the arithmetic of $\tau(n)$ and of multiple divisor sums.

Proceeding forward to the stated purpose of the section, we might consider representations of this form in terms of the size of the $\tau$-function. Using the notation $M^{\mathrm{sym}}_{\vec{k}}(n)$ for the coefficients of $\mathcal{U}^{\mathrm{sym}}_{\vec{k}}(q)$, we then obtain the following result immediately from our quasi-shuffle calculations.

\begin{corollary} \label{C: Lehmer criterion}
    Lehmer's conjecture is true if and only if for all primes $p$, we have
    \begin{align*}
        &\dfrac{200}{3} M_{(7,3)}^{\mathrm{sym}}(p) - 147 M_{(5,5)}^{\mathrm{sym}}(p) \not = \dfrac{1}{396} \sigma_{11}(p) - \dfrac{5}{6} \sigma_{7}(p) + \dfrac{137}{36} \sigma_{5}(p) + \dfrac{19}{9} \sigma_{3}(p) - \dfrac{1205}{198} \sigma(p). 
    \end{align*}
    In other words, if we define the polynomial
    \begin{align*}
        Q(x) := \dfrac{1}{396} x^{11} - \dfrac{5}{6} x^7 + \dfrac{137}{36} x^5 + \dfrac{19}{9} x^3 - \dfrac{1205}{198} x - 1,
    \end{align*}
    then Lehmer's conjecture is true if and only if
    \begin{align*}
        \dfrac{200}{3} M_{(7,3)}^{\mathrm{sym}}(p) - 147 M_{(5,5)}^{\mathrm{sym}}(p) \not = Q(p)
    \end{align*}
    for any prime $p$.
\end{corollary}

\section{Quadratic forms} \label{S: Quadratic forms}

In the previous proofs, the constructive approaches have focused on $\Gamma(N)$ for $N \geq 3$, with the case of $\Gamma(2)$ following by trace arguments. We now give a brief discussion of some calculations with $q$-multiple zeta values in levels 2 and 4 with connections to quadratic forms. We note that a similar treatment of quasimodular forms of level two could be given in terms of the two families of Eisenstein series $F_k^{(1)}(z)$ and $F_k^{(2)}(z)$ treated by Kaneko and Zagier in \cite{KanekoZagier}, which have explicit relations to $q$-multiple zeta values via divisor sums. We focus instead on theta functions since this connection is less clear a priori.

To begin, we characterize the algebra of modular forms of $\Gamma(2)$; see \cite[Chapter 1]{123} for details

\begin{lemma} \label{L: Level 2 generation}
    Define the modular forms
    \begin{align*}
        x(z) = \lp \sum_{n \in \Z} (-1)^n q^{n^2/2} \rp^4, \ \ \ y(z) = \lp \sum_{n \in \Z + \frac 12} q^{n^2/2} \rp^4.
    \end{align*}
    Then $\cM\lp \Gamma(2) \rp = \C[x(z),y(z)]$.
\end{lemma}

Since we already know how to produce $E_2$ using a $q$-multiple zeta value of level one (which of course is also of level 2), it follows from Proposition \ref{P: Polynomial in E_2} and Theorem \ref{T: YuanZhao Closure} that in order to produce $\cM\lp \Gamma(2) \rp$ as a subset of $\mathcal{MD}_q(2)$, it would suffice to produce $x(z)$ and $y(z)$ as elements of $\mathcal{MD}_q(2)$. This can also be further simplified by consideration of Jacobi theta functions. In particular, recall Jacobi's theta function
\begin{align*}
    \theta^4(z) = \sum_{n \in \Z} q^{n^2/2}.
\end{align*}
It is a classical result of Jacobi that $\theta^4(z) = y(z) - x(z)$, and so we need only prove that two out of three functions among $\theta^4, x$, and $y$ belong to $\mathcal{MD}_q(2)$.

We begin with an analysis of $\theta^4(z)$ via a slightly different normalization. It is a classical fact that
\begin{align*}
    \theta^4(2z) = \sum_{n \geq 0} r_4(n) q^n \in \cM_2\lp \Gamma_0(4) \rp,
\end{align*}
where $r_4(n)$ represents the number of ways one can write $n$ as a sum of four squares. Using an explicit basis for such spaces, it is well-known that
\begin{align*}
    r_4(n) = 8 \sum_{\substack{d|n \\ d \not \equiv 0 \bmod{4}}} d.
\end{align*}
Now, observing that for any $n > 0$ we have
\begin{align*}
    \sum_{c=0}^3 \mathcal{U}_{(1);(c);4}(q) = 4 \sum_{n \geq 1} \lp \sum_{\substack{d|n \\ d \equiv 0 \bmod{4}}} d \rp q^n,
\end{align*}
we can conclude that
\begin{align*}
    \theta^4(2z) = 1 + 6 \mathcal{U}_{(1);(0);4}(q) - 2 \mathcal{U}_{(1);(1);4}(q) - 2 \mathcal{U}_{(1);(2);4}(q) - 2 \mathcal{U}_{(1);(3);4}(q).
\end{align*}
Thus, we realize $\theta^4$ as a simple linear combination of $q$-multiple zeta values at level four. Noting also the straightforward identities $\mathcal{U}_{(1);(1);4}(q) + \mathcal{U}_{(1):(3);4}(q) = 4 \mathcal{U}_{(1);(1);2}(q^2)$ and $\mathcal{U}_{(1);(2c);4}(q) = \mathcal{U}_{(1);(c);2}(q)$, we obtain
\begin{align*}
    \theta^4(2z) = 4 \mathcal{U}_{(1);(0);2}(q) - 8 \mathcal{U}_{(1);(1);2}(q^2).
\end{align*}
We have therefore realized $\theta^4(2z)$ in both level 2 and 4 using $q$-multiple zeta values. Noting that $x(2z)$ can be realized from $\theta(2z)$ by twisting $q \mapsto -q$, we can also realize $x(2z)^4$ in terms of $q$-multiple zeta values of level 2 after twisting $q \mapsto -q$, which is consistent with the prediction of Theorem \ref{T: Main Theorem, Cyclotomic}. Thus, all of $\cM\lp \Gamma(2) \rp$ is produced in this manner by $q$-multiple zeta values. It would also be of interest to produce such values directly in level 2 without moving first through level 4; we leave this direction open.

We can derive additional nontrivial identities for partition sums using $q$-multiple zeta values in level 4 along with the ``modular identity" $\theta^4$ as $\theta^2 \cdot \theta^2$ by expanding out quasi-shuffle products. To accomplish this, we consider Eisenstein series for $\Gamma_0(N)$ with character. For a positive integer $k$ and Dirichlet characters $\psi, \chi$ of conductors $L$ and $R$ satisfying $\psi(-1) \chi(-1) = (-1)^k$, we may construct the classical Eisenstein series, for $\lp k, \psi, \chi \rp \not = \lp 2, \varepsilon, \varepsilon \rp$ for $\varepsilon$ the trivial character,
\begin{align} \label{Eq: Gamma_0(N) Eisenstein series}
    E_{k,\chi,\psi}(z) := c_0 + \sum_{n \geq 1} \lp \sum_{d|n} \psi(d) \chi(n/d) d^{k-1} \rp q^n \in M_k\lp \Gamma_0(RL), \psi \chi \rp,
\end{align}
where
\begin{align*}
    c_0 = \begin{cases}
        0 & L>1 \\ - \dfrac{B_{k,\psi}}{2k} & L=1,
    \end{cases}
    \hspace{0.1in} \text{where} \hspace{0.1in}
    \sum_{n \geq 0} B_{n,\psi} \dfrac{t^n}{n!} := \sum_{a=1}^{R} \dfrac{\psi(a) t e^{at}}{e^{Rt}-1}.
\end{align*}
These can, of course, be built from the $\Gamma(N)$-modular forms defined earlier using traces as well.

Now, it is well known that $\theta^2$ is modular of weight 1 for $\Gamma_0(4)$ with Nebentypus $\chi_4$ the Legendre character modulo 4, which is defined by $\chi_4(m) = 1$ if $m \equiv 1 \bmod{4}$, $\chi_4(m) = -1$ if $m \equiv 3 \bmod{4}$, and $\chi_4(m) = 0$ otherwise. Because the space of modular forms satisfying this property is one-dimensional and contains the Eisenstein series 
\begin{align*}
    E_{1,\chi_4}(z) = 4 E_{1,\varepsilon,\chi_4}(z) = 1 + 4 \sum_{n \geq 1} \lp \sum_{d|n} \chi_4(d) \rp q^n,
\end{align*}
we see that $\theta^2 = E_{1,\chi_4}$ and $\theta^4 = E_{1,\chi_4}^2$. In order to apply Theorem \ref{T: Main Theorem, Cyclotomic}, we rewrite $E_{1,\chi_4}$ in terms of $q$-multiple zeta values. Noting that for any $j$ we have by orthogonality that
\begin{align*}
    \dfrac 14 \sum_{c=0}^3 i^{-cj} \mathcal{U}_{(0);(c);4}(q) = \sum_{n \geq 1} \sum_{\substack{m|n \\ m \equiv j \bmod{4}}} q^n.
\end{align*}
Thus, we have
\begin{align*}
    \theta^2(2z) = 1 + \sum_{c=0}^3 \lp i^{-c} - i^c \rp \mathcal{U}_{(0);(c);4}(q) = 1 + 2i \lp \mathcal{U}_{(0);(3);4}(q) - \mathcal{U}_{(0);(1);4}(q) \rp,
\end{align*}
which we note also corresponds to the fact that the number of representations as a sum of squares of $n$ is the number of divisors of $n$ which are $1 \pmod{4}$ minus the number which are $3 \pmod{4}$, and thus by squaring
\begin{align*}
    \theta^4(z) &= 1 + 4i \lp \mathcal{U}_{(0);(3);4}(q) - \mathcal{U}_{(0);(1);4}(q) \rp - 4 \lp \mathcal{U}_{(0);(3);4}(q) - \mathcal{U}_{(0);(1);4}(q) \rp^2 \\ &= 8 \mathcal{U}_{(0);(1);4}(q) \mathcal{U}_{(0);(3);4}(q) - 4 \mathcal{U}_{(0);(1);4}(q)^2 - 4 \mathcal{U}_{(0);(3);4}(q)^2 + 4i\lp \mathcal{U}_{(0);(3);4}(q) - \mathcal{U}_{(0);(1);4}(q) \rp.
\end{align*}
We can now work this out using the quasi-shuffle structure on level 4 $q$-multiple divisor sums using \eqref{Eq: Product formula}. We thus have
\begin{align*}
    \theta^4(z) &= 1 + 8 \mathcal{U}^{\mathrm{sym}}_{(0,0);(1,1);4}(q) - 4 \mathcal{U}^{\mathrm{sym}}_{(1);(1);4}(q) + 4i\lp \mathcal{U}_{(0);(3);4}(q) - \mathcal{U}_{(0);(1);4}(q) \rp
\end{align*}

Thus, we obtain an interesting and non-trivial relation between $q$-multiple zeta values from the simple fact that $\theta^4 = \theta^2 \theta^2$, which also descends to a corresponding identity between colored multiple zeta values by comparing the pole structures of each identity near $q=1$. It would be interesting to analyze identities between $q$-multiple zeta values in light of similar studies of identities between Eisenstein series \cite{KMR}.

\section{Open Questions} \label{S: Questions}

We discuss a number of open questions in this section regarding aspects of the relationship between $q$-multiple zeta values and aspects of the theory of modular forms to which our main results do not directly speak.

\subsection{Congruences for quasimodular forms}

One of the predominant themes in both partition theory and the theory of modular forms are arithmetic properties of Fourier coefficients, and especially of coefficients in arithmetic progressions. In this area of inquiry, one of the central tools is Sturm's theorem, which has been mentioned loosely several times in the paper and which reduces many problems to a finite calculation. The essence of Sturm's theorem asserts the existence of a nice basis $f_0, \dots, f_{d-1} \in \Z[[q]]$ of forms satisfying $f_j = q^j + O\lp q^d \rp$. Informally, this result says that a modular form for $\Gamma$ of weight $k$ is determined exactly by its first $d = \dim\lp \mathcal{M}_k\lp \Gamma, \mathbb{C} \rp \rp$ coefficients. Because these coefficients are integral in the relevant algebraic number field, this result also implies that all coefficients of the form are divisible by a given integer $M$ if and only if each of the first $d$ coefficients are. This result is immensely useful in the verification of congruences for coefficients of specific modular forms.

Through works such as \cite{AAT,AOS,BCIP}, there is growing evidence that $q$-multiple zeta values will possess an interesting theory of congruences; indeed, these papers establish using Serre's theorem of $p$-adic modular forms the existence of infinitely many such congruences for many $q$-multiple zeta values. However, the arguments do not use the Sturm method, because such a method has not been developed for quasimodular forms. In fact, preliminary computer calculations indicate that such a result should fail for all large weights. We therefore pose the following open question\footnote{The author first heard this problem suggested in a talk of Ono in the Michigan Tech Seminar in Partition Theory, $q$-Series and Related Topics, but has never seen this problem written down explicitly.}

\begin{question}
    Let $f$ be a (mixed weight) quasimodular form of largest weight $k$ for a congruence subgroup $\Gamma$. Can one explicitly produce an integer $d$ such that the first $d$ coefficients of $f$ determine $f$ entirely?
\end{question}

Because preliminary computations appear to predict that a suitable $\Z$-basis may not exist in large weights, we leave a second related question.

\begin{question}
    Can we utilize $q$-multiple zeta values at level $N$ to derive Sturm-type theorems for quasimodular forms at level $N$?
\end{question}

\subsection{Weakly holomorphic modular forms}

Let $p(n)$ be the function that counts the number of unrestricted partitions of $n$. It is noted in \cite[Remark 2.1]{BachmannKuhn} that if $(0)^k = (0,0,\dots,0) \in \Z^k$, then 
\begin{align*}
    \mathcal U_{(0^k)}(q) = \sum_{0 < n_1 < n_2 < \dots < n_a} \dfrac{q^{n_1 + n_2 + \dots + n_k}}{\lp 1 - q^{n_1} \rp \lp 1 - q^{n_2} \rp \cdots \lp 1 - q^{n_k} \rp}
\end{align*}
generates partitions into exactly $k$ part sizes, and therefore
\begin{align} \label{Eq: Partition sum identity}
    \sum_{n \geq 0} p(n) q^n = \sum_{k \geq 1} \mathcal{U}_{(0^k)}(q).
\end{align}
Such arguments fit very naturally in partition theory as presented in well known texts such as \cite{AndrewsBook} covering partitions. We also recall the generating function relation
\begin{align*}
    \sum_{n \geq 0} p(n) q^n = \dfrac{1}{\lp q;q \rp_\infty} = \dfrac{q^{1/24}}{\eta(q)},
\end{align*}
where $\lp a;q \rp_\infty := \prod_{n \geq 0} \lp 1 - aq^n \rp$ and $\eta(q) := q^{1/24} \prod_{n \geq 1} \lp 1 - q^n \rp$ is Dedekind's $\eta$-function, which is a famous example of a modular form of weight 1/2. Thus, the identity \eqref{Eq: Partition sum identity} realizes the weakly holomorphic modular form $\eta^{-1}(q)$ in terms of an infinite linear combination of $q$-multiple zeta values.

Further work has demonstrated that this is not an isolated phenomenon. The next step in this direction is due to Ono and Singh \cite{OnoSingh}, who in turn were pursuing ideas relating to approximations of weakly holomorphic modular forms by MacMahon-type series \cite{AOS,BCIP}. In particular, they showed in \cite[Theorem 1.1]{OnoSingh} that for the 3-colored partition function $p_3(n)$ (counting the number of partitions of $n$ into parts which may take one of three colors) and for each $k \geq 0$ that
\begin{align} \label{Eq: Weakly hol series 1}
    \sum_{n \geq 0} p_3(n) q^n = \dfrac{1}{\eta(q)^3} = q^{- \frac{k^2 + k}{2}} \sum_{m \geq k} \binom{2m+1}{m+k+1} \mathcal{U}_{(1^m)}(q),
\end{align}
where $(1^m)$ denotes a vector of all 1's of length $m$, and in \cite[Theorem 1.2]{OnoSingh} that
\begin{align} \label{Eq: Weakly hol series 2}
    \dfrac{1}{\lp q^2; q^2 \rp_\infty \lp q;q^2 \rp_\infty^2} = q^{-k^2} \sum_{m \geq k} \binom{2m}{m+k} \mathcal{U}_{(2^m);(1)}(q),
\end{align}
with residues taken modulo 2. In \cite{JPS}, formulas are derived which realize additional weakly holomorphic modular forms as products of infinite series of the type found in \eqref{Eq: Weakly hol series 1} and \eqref{Eq: Weakly hol series 2}, which by the quasi-shuffle property can also be realized as single series, and they also derive similar kinds of formulas for infinite products which are rational functions times weakly holomorphic modular forms. We thus pose the following question.

\begin{question}
    Which eta quotients can be computed by infinite series of shifted $q$-multiple zeta values? Which weakly holomorphic modular forms can be computed this way?
\end{question}

\subsection{Traces of partition Eisenstein series}

Recall the Atkin--Garvan crank moments
\begin{align*}
	C_j(q) = \sum_{n \geq 0} M_j(n) q^n = \sum_{n \geq 0} \lp \sum_{m \in \Z} m^j M(m,n) \rp q^n.
\end{align*}
It has been shown in Section \ref{S: Partition cranks} that the theory built around the Atkin--Garvan cranks fits very naturally into the $q$-multiple zeta framework. There is another framework of a similar style which has recently been introduced into which this theory fits; these are the {\it traces of partition Eisenstein series} introduced by Amdeberhan, Ono and Singh \cite{AGOS,AOS}. A single {\it partition Eisenstein series} is defined from a partition $\lambda$ (in multiplicative notation) by the rule
\begin{align*}
	\lambda = \lp 1^{m_1}, 2^{m_2}, \dots, k^{m_k} \rp \mapsto G_\lambda(z) = G_2(z)^{m_1} G_4(z)^{m_2} \cdots G_{2k}(z)^{m_k},
\end{align*}
which is a quasimodular form of weight $2|\lambda|$. The traces of such series are then defined by a function $\phi$ on partitions by the formula
\begin{align*}
	\mathrm{Tr}_n\lp \phi \rp := \sum_{\lambda \vdash n} \phi(\lambda) G_\lambda(z).
\end{align*}
In \cite{AGOS,AOS2}, the authors produce several examples of formulas for quasimodular forms expressed in this manner, including some from Ramanujan's lost notebook and these Atkin--Garvan crank moments. Because partition Eisenstein series naturally involve a structured by complicated convolution sum, it would be natural to ask whether by introducing the structure of quasi-shuffle products might in certain cases produce particularly elegant formulas for quasimodular forms. Such a pursuit would naturally require a resolution to the following question.

\begin{question}
    Can the quasi-shuffle structure of level one $q$-multiple zeta values be used to produce an elegant formula for the partition Eisenstein series $G_\lambda(z)$?
\end{question}

\subsection{Further questions}

As the intersection between the fields of $q$-multiple zeta values, partitions and quasimodular forms is extremely new, it ought to be beneficial to explore a number of new directions in order to fully understand and utilize these connections. These new directions, we are hopeful, will interface well with the already numerous directions which exist in the areas of multiple zeta values, including connections with algebraic geometry, combinatorics, and mathematical physics. In particular, the full algebra of $q$-multiple zeta values, although it has known quasi-shuffle structure, is not well-understood from a number-theoretic perspective; such an understanding could have potential implications, for example, in attempts to prove congruences for the Fourier coefficients of quasimodular forms as in \cite{AOS} or in many other number-theoretic problems.

We therefore leave several additional open-ended questions which we hope will lead to more beautiful mathematics.

\begin{question}
    Do other natural families of $q$-series that appear often in number theory (such as $q$-hypergeometric series, mock modular forms, half-integral weight forms, or holomorphic quantum modular forms) also belong to the algebra of $q$-multiple zeta values? In particular, do negative weight Eisenstein series as studied in \cite{BOW} connect with $q$-multiple zeta values?
\end{question}

\begin{question}
    Can higher-dimensional generalizations of modular forms, such as Hilbert modular forms or Siegel modular forms, fit into some natural generalized $q$-multiple zeta value theory?
\end{question}

\begin{question}
    What are the arithmetic and asymptotic properties of the Fourier coefficients of $q$-multiple zeta values? Are there natural families whose coefficients are well-behaved asymptotically and arithmetically?
\end{question}

\section*{Statements and Declarations}

The author declares that are no competing interests.

\section*{Data Availability}

There is no data associated with this manuscript.

\end{document}